\documentclass{article}    

\usepackage{amsmath}        
\usepackage{amsfonts}       
\usepackage{graphicx}       
\usepackage{xcolor}          
\usepackage{subfig}         
\usepackage{amsthm}

\usepackage{fullpage}

\newcommand{\R}{\mathbb{R}}
\newcommand{\N}{\mathbb{N}}
\newcommand{\dd}{\mathrm{d}}
\newcommand{\kC}{\mathcal{C}}
\newcommand{\kI}{\mathcal{I}}
\newcommand{\kJ}{\mathcal{J}}
\newcommand{\kT}{\mathcal{T}}
\newcommand{\kF}{\mathcal{F}}

\newcommand{\T}[1]{#1^T}

\newcommand{\abs}[1]{\left|{#1}\right|}
\newcommand{\uvoz}[1]{``#1"}

\newtheorem{theorem}{Theorem}
\newtheorem{lemma}{Lemma}

\newtheorem{remark}{Remark}
\newtheorem{definition}[lemma]{Definition}
\newtheorem{example}{Example}

\begin{document}

\title{Discontinuous Galerkin method for macroscopic traffic flow models on networks\thanks{The work of L. Vacek is supported by the Charles University, project GA UK No. 1114119. The work of V. Ku\v{c}era is supported by the Czech Science Foundation, project No. 20-01074S.}}



\author{%
{\sc
Luk\'a\v s Vacek\thanks{Corresponding author. Email: lvacek@karlin.mff.cuni.cz}} \\[2pt]
Charles University, Faculty of Mathematics and Physics\\
Sokolovsk\'{a} 83, Praha 8, 186\,75, Czech Republic\\[6pt]
{\sc and}\\[6pt]
{\sc V\'aclav Ku\v cera}\thanks{Email: kucera@karlin.mff.cuni.cz}\\[2pt]
Charles University, Faculty of Mathematics and Physics\\
Sokolovsk\'{a} 83, Praha 8, 186\,75, Czech Republic
}



\maketitle

\begin{abstract}
In this paper, we describe a numerical technique for the solution of macroscopic traffic flow models on networks of roads. On individual roads, we consider the standard Lighthill-Whitham-Richards model which is discretized using the discontinuous Galerkin method along with suitable limiters. In order to solve traffic flows on networks, we construct suitable numerical fluxes at junctions based on preferences of the drivers. We prove basic properties of the constructed numerical flux and the resulting scheme and present numerical experiments,  including a junction with complicated traffic light patterns with multiple phases. Differences with the approach to numerical fluxes at junctions from \v Cani\'c et al., 2015, are discussed and demonstrated numerically on a simple network.
\end{abstract}

\section*{Introduction}
This paper deals with the numerical solution of traffic flows on networks of roads. The mathematical description of the flow of vehicles (cars) on roads can basically be divided into three approaches based on the level of description -- \emph{microscopic} (where we track every individual vehicle), \emph{mesoscopic} (analogous to the kinetic Boltzmannian approach for gas dynamics) and \emph{macroscopic}, cf. \cite{vanWageningen-Kessels}. We will deal with the latter, macroscopic approach, where traffic on a road is viewed as a single moving continuum, usually described by its point-wise density and velocity. The resulting mathematical description can then be viewed as analogous to the equations of gas dynamics. Since the basic property of traffic flow is the conservation of the total number of vehicles, first order hyperbolic equations, or conservation laws, naturally arise in this context, cf. \cite{Networks}.

We will be concerned with the classical Lighthill-Whitham-Richards (LWR) model, which is a scalar nonlinear first order hyperbolic equation for traffic density, cf. \cite{Lecture_notes_Jungel}, \cite{Mathematical_Framework} and \cite{Networks} for an overview. The LWR model is supplemented by a so-called \emph{fundamental diagram}, which relates traffic density and traffic flow in homogeneous traffic, cf. \cite{Greenshields35}. Thus the LWR model is in fact a whole class of models depending on the choice of the fundamental diagram. 

In this paper, we consider LWR models on networks of roads, cf. \cite{Networks}, \cite{RKDG}. On each individual road, traffic is described simply by the equation arising from the LWR model. At junctions, it is necessary to specify how traffic will be divided between incoming and outgoing roads. This is done according to the \emph{traffic distribution matrix} at each junction, which is based on the drivers' preferences. It is then necessary to express the traffic flows from individual incoming to individual outgoing roads, cf. \cite{Networks} for details.

Since we deal with first order hyperbolic problems, the natural choice of numerical method is the \emph{discontinuous Galerkin} (DG) method, which has become a robust, well understood and popular numerical method for such problems in the past decades  \cite{Cockburn_Karniadakis_Shu}, \cite{DG}. The DG method can be viewed as a combination of the finite element and finite volume methods, which is inherently of arbitrary order of accuracy. This method uses discontinuous, piecewise polynomial approximations on a partition, with the assumption of global continuity being replaced by a weaker form, using the numerical diffusion of a numerical flux function at interfaces between elements of the partition. Thus this performs well on problems with discontinuous solutions or solutions with steep internal or boundary layers, such as those considered in this paper. 

While the DG method in itself is rather well understood with a solid theoretical and practical background, cf. \cite{DG}, \cite{Stabilization}, the application of the method on networks is much less standard, \cite{RKDG}. The main problem lies in the construction of numerical fluxes (or even exact Riemann solvers) at nodes (junctions) of the network. This construction must somehow reflect the preferences of the drivers when deciding which way to turn at the junction. This has been done in \cite{Networks}, based on a traffic distribution matrix and the assumption that drivers maximize the total traffic flow through the junction. The disadvantage of this approach is that the construction of the fluxes requires the solution of a Linear Programming problem, which is rather complicated in general, although it can be solved analytically in simple cases \cite{RKDG}. In this paper, we present a simpler alternative construction of the numerical fluxes at junctions, which has the advantage that it is given by an explicit formula for any type of junction. When comparing the two approaches, that of \cite{RKDG} corresponds to single-lane roads with a strict enforcement of a priori traffic distribution, while the presented approach corresponds to having dedicated turning-lanes and/or flexibility of the drivers' preferences in extreme situation such as congestions. Moreover, the presented construction of the traffic flux at junctions allows the simulation of arbitrary traffic light combinations, while that of \cite{RKDG} only allows full green or full red lights on incoming roads. We prove basic properties of the proposed numerical fluxes and DG scheme, discuss the differences between our approach and \cite{RKDG} and present numerical experiments.

The paper is organized as follows. Section \ref{Traffic_Flow} gives a necessary background on macroscopic traffic flow models and traffic flows on networks. In Section \ref{DG_intro}, we present the basic DG scheme on a single domain (road), discuss the numerical flux, limiters and implementation. In Section \ref{Sec_DG_method_on_networks}, we define the DG method on networks, construct the numerical fluxes at junctions, prove basic properties of the resulting DG scheme and discuss the interpretation of the presented construction. Finally, Section \ref{sec_Numerical_results} contains numerical results, including a comparison of the presented approach and that of \cite{RKDG} on a simple network, and the simulation of traffic flow through a junction with complicated traffic light patterns with multiple phases.

\section{Macroscopic traffic flow models}\label{Traffic_Flow}
\subsection{Fundamental quantities and models}
We begin with the mathematical description of macroscopic vehicular traffic, cf. \cite{Lecture_notes_Jungel}, \cite{Mathematical_Framework} and \cite{vanWageningen-Kessels} for details. First, we consider a single road described mathematically as a one-dimensional interval. In the basic macroscopic models, traffic flow is described by three basic fundamental quantities -- \emph{traffic flow $Q$, traffic density $\rho$} and \emph{mean traffic flow velocity $V$}.

The \emph{traffic flow} $Q(x,t)$ determines the number of vehicles passing through a point $x$ on the road within an infinitesimal interval containing the time instant $t$. Traffic flow is measured in vehicles per second and can be formally defined as
\begin{equation}
Q(x,t)=\lim_{\substack{|I_t|\to 0\\ t\in I_t}}\dfrac{N_t(x,I_t)}{|I_t|},
\end{equation}
where $N_t(x,I_t)$ is the number of vehicles passing through the point $x$ within the time interval $I_t$ containing $t$. Traffic flow can be measured from real traffic data.

\emph{Traffic density} $\rho(x,t)$ determines the number of vehicles inside an infinitesimal spatial interval containing $x$, at time $t$. Its unit is cars per meter and it can formally defined as 
\begin{equation}
\rho(x,t)=\lim_{\substack{|I_x|\to 0\\ x\in I_x}}\dfrac{N_x(I_x,t)}{|I_x|},
\end{equation}
where $N_x(I_x,t)$ is the number of vehicles in the interval $I_x$ at the time $t$. Similarly as $Q$, traffic density can be measured from real traffic data.

Finally, the \emph{mean traffic flow velocity} $V(x,t)$ is defined simply as
\begin{equation}\label{def_velocity}
V(x,t)=\dfrac{Q(x,t)}{\rho(x,t)},
\end{equation} 
its unit being meters per second. We note that in general this quantity is not the velocity of a single car. Instead, $V$ can be viewed as the group or average velocity in the neighborhood $x$, which can differ from the velocity of individual cars.

The basic governing equation of traffic flow is derived using the assumption that the number of cars in a segment $[x_1,x_2]$ of the road cars changes only due to the flux through the endpoints, i.e. 
\begin{equation}
\dfrac{\dd}{\dd t}\int_{x_1}^{x_2}\rho(x,t)\,\dd x = Q(x_1,t)-Q(x_2,t).
\end{equation}
Writing the right-hand side as an integral, expressing $Q$ using (\ref{def_velocity}) and eliminating the integral gives the conservation law for $\rho$ in the form
\begin{equation}\label{Conservation_law}
\dfrac{\partial}{\partial t}\rho(x,t)+\dfrac{\partial}{\partial x}\big(\rho(x,t)V(x,t)\big)=0.
\end{equation}
Equation (\ref{Conservation_law}) must be supplemented by an initial condition
and appropriate boundary conditions which we will treat in detail in the case of networks of roads.

\subsection{Lighthill--Whitham--Richards model}

\begin{figure}[b!]\centering
\subfloat[Velocity--density diagrams.]{\includegraphics[height=1.7in]{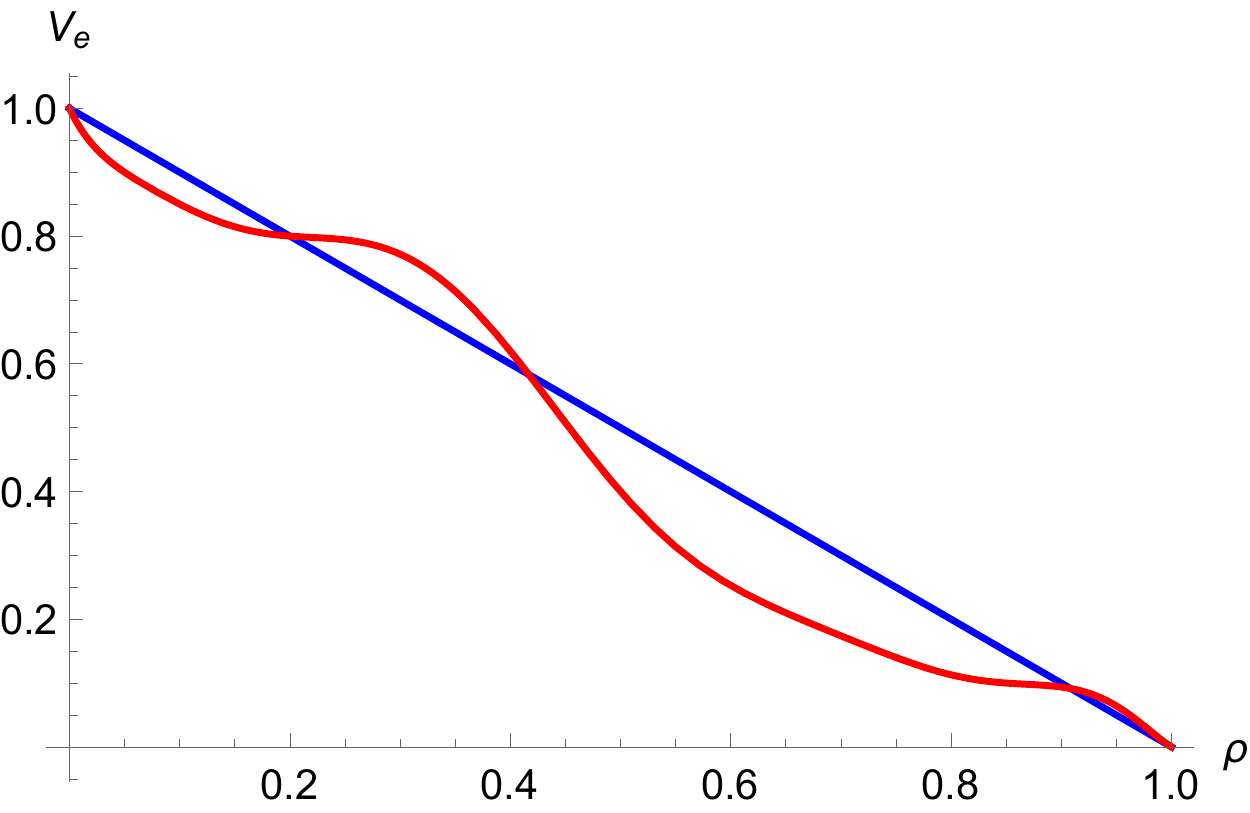}}
\hspace{5pt}
\subfloat[Flow--density diagrams.]{\includegraphics[height=1.7in]{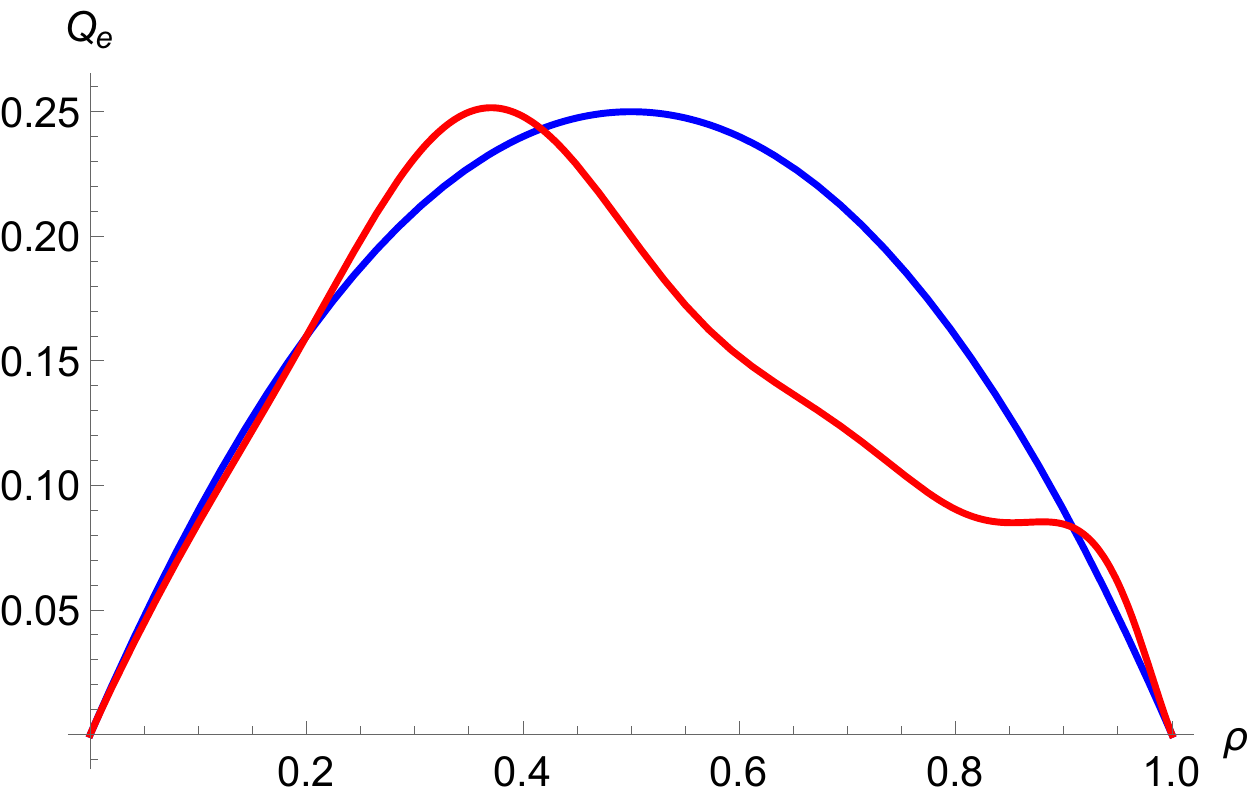}}
\caption{Examples of fundamental diagrams.}
\label{obr_examples_fundamental}
\end{figure}

\begin{figure}[t!]\centering
\subfloat[Velocity--density diagram.]{\includegraphics[height=1.59in]{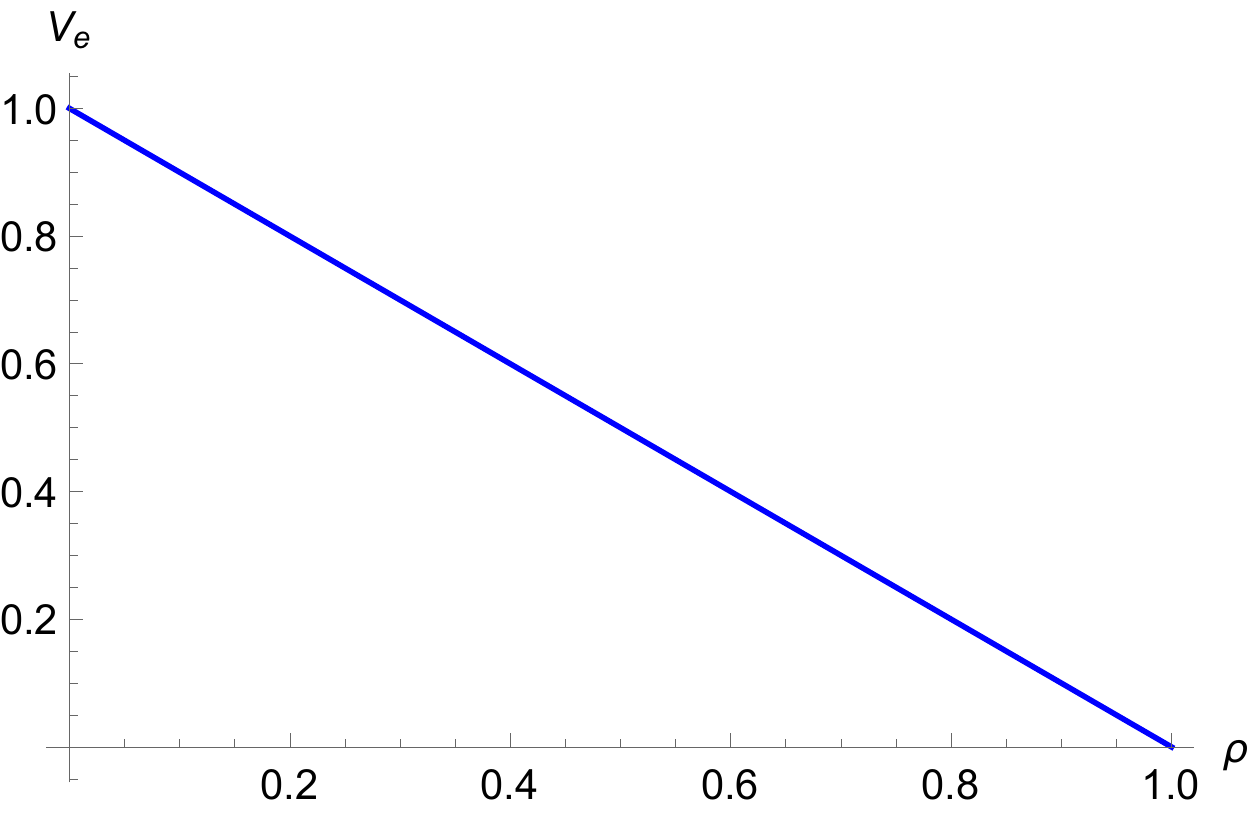}}
\hspace{5pt}
\subfloat[Flow--density diagram.]{\includegraphics[height=1.59in]{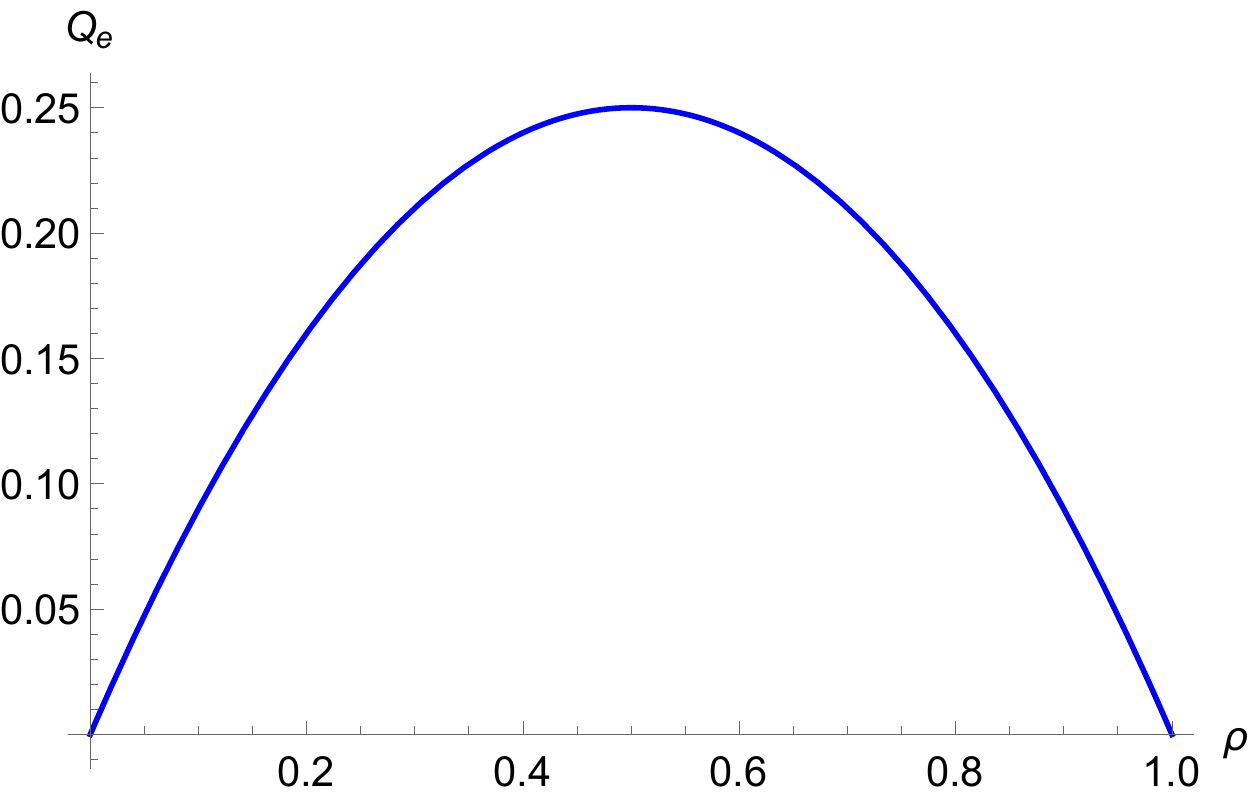}}
\caption{Fundamental diagrams of the Greenshields model.}
\label{obr12}
\end{figure}

Equation (\ref{Conservation_law}) is underdetermined, as there is a single equation for two unknowns. Thus we need to supply another equation or relation between the variables. Greenshields described a relation between traffic density and traffic flow in the paper \cite{Greenshields35}. He realized that traffic flow is a function which depends only on one variable in homogeneous traffic (traffic with no changes in time and space). This one variable is traffic density. This implies that even mean traffic flow velocity depends only on traffic density. Let us denote the equilibrium quantity of homogeneous traffic as $Q_e$, derived from $Q$, and the equilibrium quantity $V_e$ derived from $V$. Following (\ref{def_velocity}),  these equilibrium quantities corresponding to homogeneous traffic satisfy:
\begin{equation}\label{equilib1}
Q_e(\rho)=\rho V_e(\rho).
\end{equation}
In general it is assumed that $V_e$ is a nonincreasing function of $\rho$. Thus, maximal equilibrium traffic flow is attained at a certain density value. The relationship between the $\rho$ and $V_e$ is described by the \emph{fundamental diagram}. Typical fundamental diagrams are shown in Figure \ref{obr_examples_fundamental} -- the blue line in both figures represent the Greenshields model described below.

The Lighthill--Whitham--Richards model (abbreviated LWR) is an approach where we use the equilibrium velocity $V_e$ in equation (\ref{Conservation_law}) resulting in the equation
\begin{equation}\label{LWR_problem}
\begin{split}
\rho_t+\left(Q_e(\rho)\right)_x=0,&\qquad x\in\R,\ t>0,\\
\rho(x,0)=\rho_0(x),&\qquad x\in\R,
\end{split}
\end{equation}
where $Q_e(\rho)$ is the equilibrium traffic flow derined by (\ref{equilib1}).
Equation (\ref{LWR_problem}) belongs to the class of \emph{nonlinear first order hyperbolic equations}. 

There are many different proposals for the equilibrium velocity $V_e$ derived from real traffic data, cf. \cite{Mathematical_Framework}. Here we present only two basic models.

\paragraph{Greenshields model}
This model uses a linear relationship between traffic density and equilibrium traffic velocity:
\[
V_e(\rho)=v_{\max}\left( 1-\dfrac{\rho}{\rho_{\max}}\right),
\]
where $v_{\max}$ is the maximal velocity and $\rho_{\max}$ is the maximal density. We can see the fundamental diagram in Figure \ref{obr12}, where $v_{\max}=\rho_{\max}=1$.

\paragraph{Greenberg model}
This model uses the equilibrium velocity given by
\[
V_e(\rho)=v_{\max}\ln\left(\dfrac{\rho_{\max}}{\rho}\right),
\]
Thus, traffic can overcome the maximal velocity $v_{\max}$. We can see the fundamental diagram in Figure \ref{obr13}, where $v_{\max}=\rho_{\max}=1$.
\begin{figure}[t!]\centering
\subfloat[Velocity--density diagram.]{\includegraphics[height=1.59in]{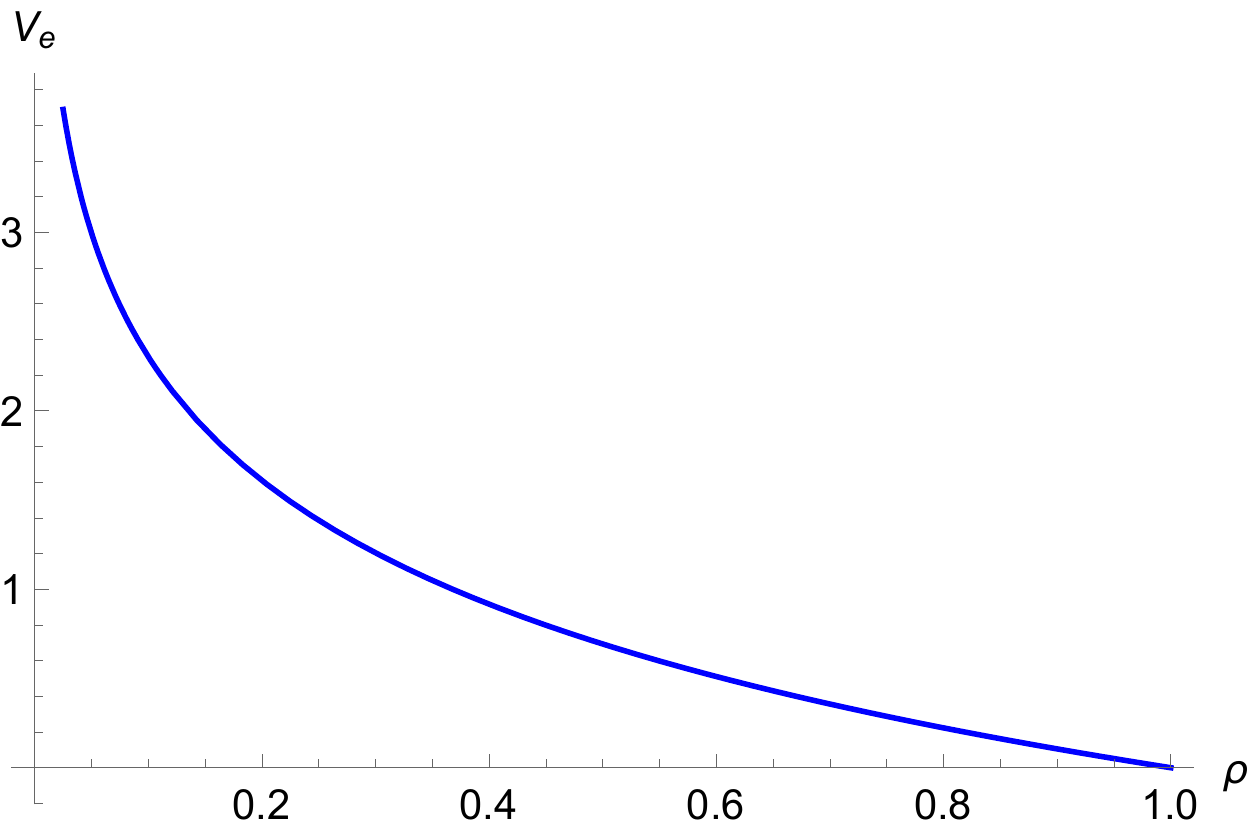}}
\hspace{5pt}
\subfloat[Flow--density diagram.]{\includegraphics[height=1.59in]{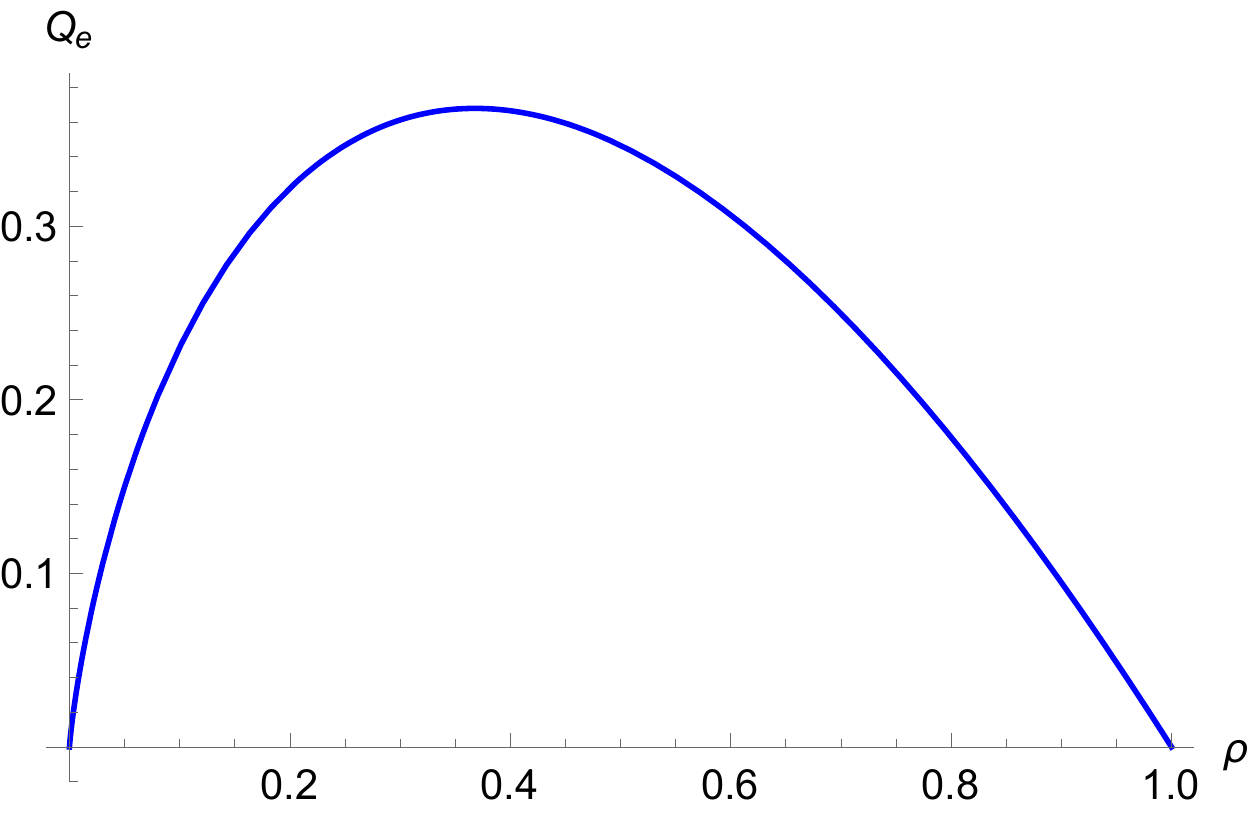}}
\caption{Fundamental diagrams of the Greenberg model.}
\label{obr13}
\end{figure}

\subsection{Traffic flows on networks}\label{Network}
In this section, we introduce the basic concepts and notation describing traffic flows on networks. We refer the reader to \cite{Networks} for a more complete treatment of the topic. 

We consider a network represented by a directed graph. The graph is a finite collection of directed edges, connected together at vertices. Each vertex has a finite set of \emph{incoming edges} and \emph{outgoing edges}.
\begin{definition}[Network]\label{Def_network}	
We define a \emph{network} as a couple $(\kI,\kJ)$, where $\kI=\lbrace I_n\rbrace_{n=1}^N$ is a finite set of edges and $\kJ=\lbrace J_m\rbrace_{m=1}^M$ is a finite set of vertices. Each edge $I_n$ is represented by an interval $[a_n,b_n]\subseteq[-\infty,\infty],\ n=1,\ldots,N$. Each vertex $J_m$ is a union of two non--empty subsets $\text{Inc}(J_m)$ and $\text{Out}(J_m)$ of $\lbrace 1,\ldots,N\rbrace$ representing \emph{incoming} and \emph{outgoing edges}, respectively. We assume the following:
\begin{itemize}
\item[(i)] For all $J_i,J_j\in\kJ,\ i\neq j:\text{Inc}(J_i)\cap\text{Inc}(J_j)=\emptyset$ and $\text{Out}(J_i)\cap\text{Out}(J_j)=\emptyset$.
\item[(ii)] If $i\notin\cup_{J\in\kJ}\text{Inc}(J)$, $i\in\lbrace 1,\ldots,N\rbrace$, then $b_i=\infty$ and if $i\notin\cup_{J\in\kJ}\text{Out}(J)$, $i\in\lbrace 1,\ldots,N\rbrace$, then $a_i=-\infty$. Moreover, for all $i\in\lbrace 1,\ldots,N\rbrace: i\in\cup_{J\in\kJ}\text{Inc}(J)$ or $i\in\cup_{J\in\kJ}\text{Out}(J)$.
\end{itemize}
\end{definition}

Condition (i) states that each edge can be incoming for at most one vertex and outgoing for at most one vertex. Condition (ii) states that edges that are connected to only one vertex extend to $\pm\infty$. Of course in practice artificial inflow/outflow boundaries are introduced in the numerical solution. We can see an example in Figure \ref{obr15}.

\begin{figure}[b!]\centering
\includegraphics[height=1.8in]{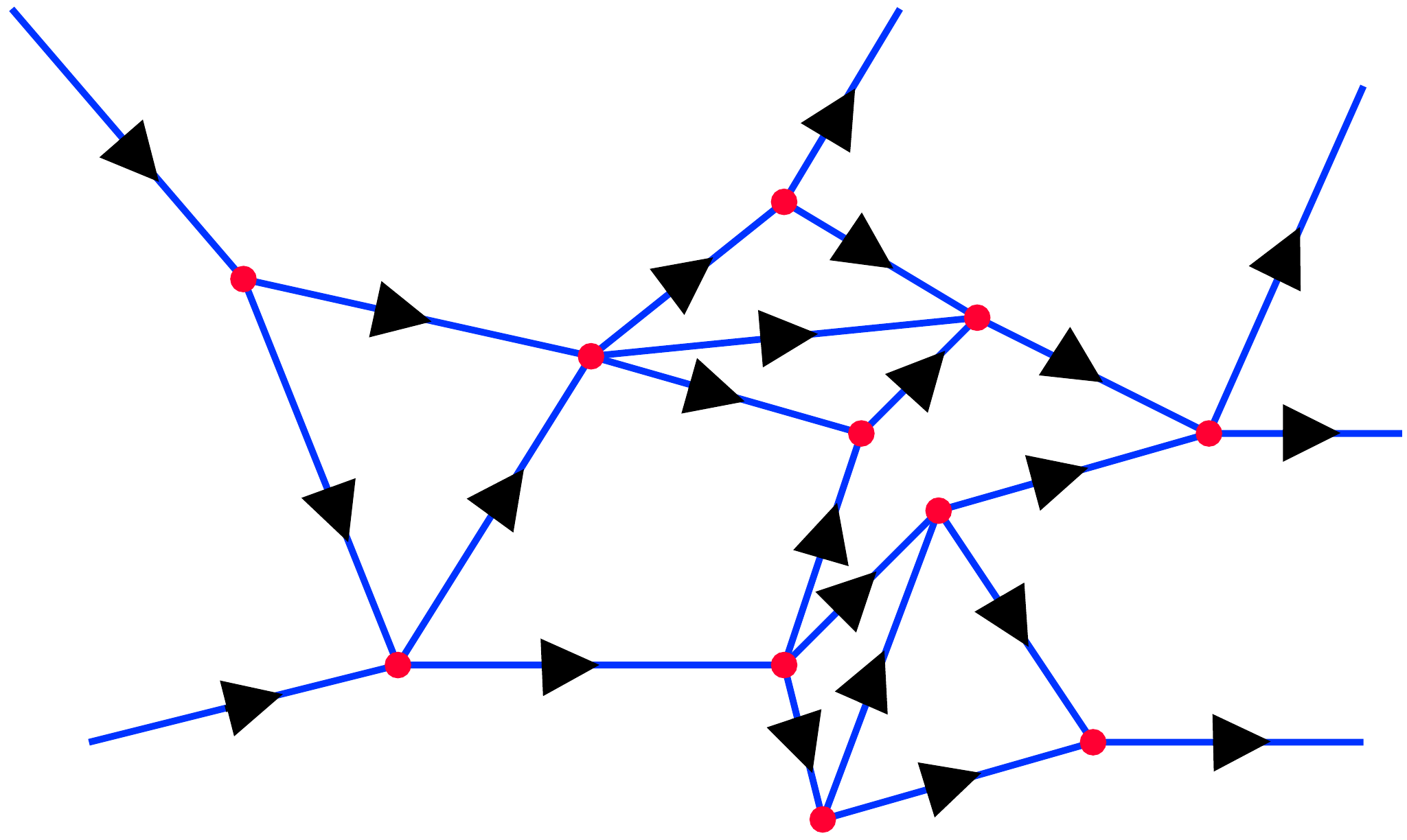}
\caption{Example of a network.}
\label{obr15}
\end{figure}

As we are dealing with traffic flows described by LWR models, we assume that the traffic on edge number $i\in\lbrace 1,\ldots,N\rbrace$ is described by
\begin{equation}\label{LWR_network}
\begin{split}
(\rho_i)_t+\left(Q_e(\rho_i)\right)_x=0,&\qquad x\in(a_i,b_i),\ t>0,\\
\rho_i(x,0)=\rho_{0,i}(x),&\qquad x\in(a_i,b_i),
\end{split}
\end{equation}
where $\rho_i:(a_i,b_i)\times[0,\infty)\to\R$ is the traffic density on the $i$-th edge (road).

What remains is to describe the behavior of traffic at junctions. For this purpose it is sufficient to first consider a single vertex (junction) and its incoming and outgoing roads for simplicity. The resulting considerations can then be applied to each vertex of the general network separately. 

We consider a network $(\kI,\kJ)$ and fix a vertex $J\in\kJ$ for which we assume that $\text{Inc}(J)=\lbrace 1,\ldots,n\rbrace$ and $\text{Out}(J)=\lbrace n+1,\ldots,n+m\rbrace$. We define the spatial limits of traffic densities on individual roads at the common vertex $J$ as
\[
\rho_i^{(L)}(b,t):=\lim_{x\to b_-}\rho_i(x,t)\quad\text{and}\quad \rho_j^{(R)}(a,t):=\lim_{x\to a_+}\rho_j(x,t)
\]
for all $i=1,\ldots,n$ and $j=n+1,\ldots,n+m$. Now we are ready to present the definitions of solution at junctions.
\begin{definition}[Traffic solution at a junction]\label{def_Traffic_solution_at_junction}
Let $J$ be a junction with incoming roads $I_1,\ldots,I_n$ and outgoing road $I_{n+1},\ldots,I_{n+m}$. Then we define a weak solution at $J$ as a collection of functions $\rho_l:I_l\times[0,\infty)\rightarrow\R$, $l=1,\ldots,n+m$ such that
\[
\sum_{l=1}^{n+m}\left(\int_{a_l}^{b_l}\int_0^\infty\left( \rho_l\dfrac{\partial\varphi_l}{\partial t}+Q_e(\rho_l)\dfrac{\partial\varphi_l}{\partial x}\right)\dd t\dd x\right)=0
\]
holds for every $\varphi_l\in\kC_0^1([a_l,b_l]\times [0,\infty))$, $l=1,\ldots,n+m$, that are also smooth across the junction, i.e.
\[
\varphi_i^{(L)}(b_i,\cdot)=\varphi_j^{(R)}(a_j,\cdot),\qquad\left(\dfrac{\partial\varphi_i}{\partial x}\right)^{(L)}(b_i,\cdot)=\left(\dfrac{\partial\varphi_j}{\partial x}\right)^{(R)}(a_j,\cdot),
\]
for all $i\in\lbrace 1,\ldots,n\rbrace$ and $j\in\lbrace n+1,\ldots,n+m\rbrace$.
\end{definition}

The basic property of the weak solution from Definition \ref{def_Traffic_solution_at_junction} is that it satisfies the Rankine-Hugoniot condition which is essentially the conservation of vehicles at the junction.
\begin{lemma}
Let $\rho=\T{(\rho_1,\ldots,\rho_{n+m})}$ be a weak solution at the junction $J$ such that each $x\rightarrow\rho_i(x,t)$ has bounded variation. Then $\rho$ satisfies the \emph{Rankine-Hugoniot condition}
\begin{equation}\label{Junction_equation}
\sum_{i=1}^n Q_e(\rho_i^{(L)}(b_i,t))=\sum_{j=n+1}^{n+m} Q_e(\rho_j^{(R)}(a_j,t))
\end{equation}
for almost every $t>0$ at the junction $J$.
\end{lemma}
\begin{proof}
The proof is a simple application of integration by parts and can by found in  \cite[Lemma 5.1.9]{Networks}.
\end{proof}

Definition \ref{def_Traffic_solution_at_junction} simply enforces the conservation of vehicles at $\kJ$. However it is also necessary to take into account the preferences of drivers how the traffic from incoming roads is distributed to outgoing roads according to some predetermined coefficients.

\begin{definition}[Traffic distribution matrix]
Let $J$ be a fixed vertex with $n$ incoming edges and $m$ outgoing edges. We define a \emph{traffic distribution matrix} $A$ as
\[
A=\begin{bmatrix}
\alpha_{n+1,1} & \cdots & \alpha_{n+1,n}\\
\vdots & \vdots & \vdots \\
\alpha_{n+m,1} & \cdots & \alpha_{n+m,n}
\end{bmatrix},
\]
where $0\leq\alpha_{j,i}\leq 1$ for all $i\in\lbrace 1,\ldots,n\rbrace,j\in\lbrace n+1,\ldots,n+m\rbrace$ and
\begin{equation}\label{sum_alpha}
\sum_{j=n+1}^{n+m}\alpha_{j,i}=1
\end{equation}
holds for all $i\in\lbrace 1,\ldots,n\rbrace$.
\end{definition}

The $i^{th}$ column of $A$ describes how the traffic from the incoming road $I_i$ distributes to the outgoing roads at the junction $J$. In other words, if X is the amount of traffic coming from road $I_i$ then $\alpha_{j,i}X$ is the amount of traffic going form $I_i$ towards road $I_j$.

Based on the traffic distribution matrix, the authors of \cite{Networks} define the following admissible traffic solution at a junction, also used for numerical simulations in \cite{RKDG}.

\begin{definition}[Admissible traffic solution at a junction, following \cite{Networks}]\label{def_network_solution}
Let $\rho=\T{(\rho_1,\ldots,\rho_{n+m})}$ be such that $\rho_i(\cdot,t)$ is of bounded variation for every $t\geq 0$. Then $\rho$ is called an \emph{admissible weak solution} of (\ref{LWR_network}) related to the matrix $A$ at the junction $J$ if the following properties hold:
\begin{itemize}
\item[(i)] $\rho$ is a weak solution at the junction $J$.
\item[(ii)] $Q_e(\rho_j^{(R)}(a_j,\cdot))=\sum_{i=1}^n\alpha_{j,i}Q_e(\rho_i^{(L)}(b_i,\cdot)),$ for all $j=n+1,\ldots,n+m$.
\item[(iii)] $\sum_{i=1}^n Q_e(\rho_i^{(L)}(b_i,\cdot))$ is a maximum subject to (i) and (ii).
\end{itemize}
\end{definition}

\begin{remark}
Condition (ii) simply states that traffic from incoming roads is distributed to outgoing roads according to the traffic distribution matrix. Condition (iii) is a mathematical formulation of the assumption made in \cite{Networks}, that respecting (ii), \uvoz{drivers choose so as to maximize fluxes} through the junction.
\end{remark}

One problem with the approach of \cite{Networks} and \cite{RKDG} is that explicitly constructing the fluxes from Definition \ref{def_network_solution} requires the solution of a Linear Programming problem on the incoming fluxes. This is done in \cite{Networks} for the purposes of constructing a Riemann solver at the junction and in \cite{RKDG} for the purposes of obtaining numerical fluxes at the junction in order to formulate the DG scheme. Closed-form solutions are provided in \cite{RKDG} in the special cases $n=1,m=2$ and  $n=2,m=1$ and $n=2,m=2$. In Section \ref{Sec_DG_method_on_networks}, we present an alternative construction of fluxes at the junction which has the advantage of a simple formulation for general $n,m$. We will give an interpretation of our construction, which shows that it is more suited for certain situations, giving more realistic behavior of the drivers, than the approach from Definition \ref{def_network_solution}. We compare the two approaches in Section \ref{sec_num_exp_comparison}.

\section{Discontinuous Galerkin method}\label{DG_intro}
We discretize the governing equation (\ref{Conservation_law}) using the \emph{discontinuous Galerkin} (DG) method. This method introduced by Reed and Hill in \cite{PrvniDG} represents a robust, reliable and accurate numerical method for the solution of first order hyperbolic problems. The DG method uses discontinuous piecewise polynomial approximation of the exact solution along with a suitable weak form of the governing equations and can thus be viewed as a combination of the the finite element and finite volume methods, cf. \cite{DG}, \cite{Stabilization}. One advantage of the DG method over standard finite elements is it's robustness with respect to the Gibbs phenomenon. This occurs when a continuous approximation is used to approximate a discontinuous function -- these typically arise as solutions to nonlinear first order hyperbolic problems, such as those considered here.

In general, the DG method is described on a polygonal (polyhedral) domain $\Omega\subset\R^d$, $d\in\N$. Let $\kT_h$ be a partition of $\overline{\Omega}$ into a finite number of closed elements $K$ with mutually disjoint interiors, such that
\[
\overline{\Omega}=\bigcup_{K\in\kT_h}K.
\]
Since the traffic model is defined on a line, we consider $\Omega\subset\R$, $\Omega=(a,b)$. In the 1D case, an element $K$ is an interval $\left[ a_K,\ b_K\right]$, where $a_K$ and $b_K$ are boundary points of $K$. We set $h_K=\abs{b_K-a_K}$, $h=\max_{K\in\kT}h_K$. We denote the set of all boundary faces (points in 1D) of all elements by $\kF_h$. Further, we define the set of all inner points by
\[
\kF_h^I=\lbrace x\in\kF_h;\ x\in\Omega\rbrace 
\]
and the set of boundary points $\kF_h^B=\lbrace a,\ b\rbrace$. Obviously $\kF_h=\kF_h^I\cup\kF_h^B$.

We use a suitable weak formulation of (\ref{Conservation_law}) on the \emph{broken Sobolev space} $H^k(\Omega,\ \kT_h)\\=\lbrace v;\ v|_K\in H^k(K),\ \forall K\in\kT_h\rbrace$, where $H^k(I)$, $k\in\N$ be the Sobolev space over an interval $I$. Functions from this space will be approximated by discontinuous piecewise polynomial functions
\[
S_h=\lbrace v;\ v|_K\in P^p(K),\ \forall K\in\kT_h\rbrace,
\]
where $P^p(K)$ denotes the space of all polynomials on $K$ of degree at most $p$.

For each point $x\in\kF_h^I$ there exist two neighbours $K_x^{(L)},\ K_x^{(R)}\in\kT_h$ such that $x=K_x^{(L)}\cap K_x^{(R)}$. Every function $v\in H^k(\Omega,\kT_h)$ is generally discontinuous at $x\in\kF_h^I$. Thus, for all $x\in\kF_h^I$, we introduce the following notation:
\begin{equation*}
v^{(L)}(x)=\lim_{y\rightarrow x_-}v(y),\qquad v^{(R)}(x)=\lim_{y\rightarrow x_+}v(y),\qquad\left[ v\right]_x=v^{(L)}(x)-v^{(R)}(x).
\end{equation*}
In order to have consistent notation, in the point $x\in\kF_h^B$ we define
\begin{align*}
v^{(R)}(a)=\lim_{y\rightarrow a_+}v(y),\qquad&v(a)=-\left[ v\right]_a=v^{(L)}(a):=v^{(R)}(a),\\
v^{(L)}(b)=\lim_{y\rightarrow b_-}v(y),\qquad&v(b)=\left[ v\right]_b=v^{(R)}(b):=v^{(L)}(b).
\end{align*}
The definition of jump $\left[ v\right]_a:=-v^{(R)}(a)$ or $\left[ v\right]_b:=v^{(L)}(b)$ may seem inconsistent with the definition on interior points. This notation is used due to the integration by parts in following sections. Our notation allows us to simplify those terms.

For simplicity, if $\left[\cdot\right]_x$ appear in a sum of the form $\sum_{x\in\kF_h}\ldots$, we omit the index $x$ and write $\left[\cdot\right]$.

\subsection{First order hyperbolic problems}
We formulate the DG method for first order hyperbolic problems of the form
\begin{align}
u_t+f(u)_x=0,&\qquad x\in\Omega,\ t\in(0,T),\label{DG_Hyperbolic_problem}\\
u=u_D,&\qquad x\in\kF_h^D,\ t\in(0,T),\\
u(x,0)=u_0(x),&\qquad x\in\Omega,\label{DG_Hyperbolic_problem_initial}
\end{align}
where the Dirichlet boundary condition $u_D:\kF_h^D\times(0,T)\rightarrow\R$ and the initial condition $u_0:\Omega\rightarrow\R$ are given functions. The Dirichlet boundary condition is prescribed only on the inlet $\kF_h^D\subseteq\kF_h^B$, respecting the direction of information propagation. The function $f\in\kC^1(\R)$ is called the \emph{convective flux}. Our aim is to seek a function $u:\Omega\times(0,T)\rightarrow\R$ such that (\ref{DG_Hyperbolic_problem})--(\ref{DG_Hyperbolic_problem_initial}) is satisfied. As we have seen, problem (\ref{DG_Hyperbolic_problem}) is the main part of macroscopic equations for traffic.

In order to derive the DG formulation of (\ref{DG_Hyperbolic_problem}), we multiply by a test function $\varphi\in H^1(\Omega,\kT_h)$ and integrate over an arbitrary element $K\in\kT_h$. Then we apply integration by parts and obtain
\begin{equation}
\int_K u_t\varphi\ \dd x-\int_K f(u)\varphi'\ \dd x+f(u(b_K,t))\varphi^{(L)}(b_K)-f(u(a_K,t))\varphi^{(R)}(a_K)=0.
\end{equation}
Finally, we sum over all $K\in\kT_h$ and obtain
\[
\int_\Omega u_t\varphi\ \dd x-\sum_{K\in\kT_h}\int_K f(u)\varphi'\ \dd x+\sum_{x\in\kF_h}f(u)\left[\varphi\right] =0.
\]

We wish to approximate $u$ by a function $u_h\in H^1(\Omega,\kT_h)$ which is in general discontinuous on $\kF_h$. Thus, we need to give proper meaning to the function $f(u_h)$ in points $x\in\kF_h$. We proceed similarly as in the finite volume method and use the approximation
\begin{equation}
f(u_h)\approx H(u_h^{(L)},u_h^{(R)}),
\end{equation}
where $H(\cdotp,\cdotp)$ is a \emph{numerical flux}, cf. \cite{DG}. Finally, we define the DG solution of problem (\ref{DG_Hyperbolic_problem}).

\begin{definition}[DG solution]
The function $u_h:\Omega\times(0,T)\rightarrow\R$ is called a DG finite element solution of hyperbolic problem (\ref{DG_Hyperbolic_problem})--(\ref{DG_Hyperbolic_problem_initial}) if the following properties hold:
\begin{itemize}
\item[(i)] $u_h\in\kC^1\left(\left[ 0,T\right];S_h\right)$.
\item[(ii)] $u_h(0)=u_{h0}$, where $u_{h0}$ denotes an $S_h$ approximation of the initial condition $u_0$.
\item[(iii)] $u_h=u_D$ for all $x\in\kF_h^D,\ t\in(0,T)$.
\item[(iv)] For all $\varphi\in S_h$ and for all $t\in\left( 0,T\right)$, $u_h$ satisfies
\begin{equation}\label{DG_Weak}
\int_\Omega (u_h)_t\varphi\ \dd x-\sum_{K\in\kT_h}\int_K f(u_h)\varphi'\ \dd x+\sum_{x\in\kF_h}H(u^{(L)}_h,u^{(R)}_h)\left[\varphi\right] =0.
\end{equation}
\end{itemize}
\end{definition}

\subsection{Implementation details}
In our implementation, we use the \emph{Lax--Friedrichs} numerical flux, cf. \cite{DG}, \cite{Stabilization}. We define
\begin{equation}\label{Lax_Friedrichs_flux}
H\big( u^{(L)},u^{(R)}\big)=\dfrac{1}{2}\left( f\big(u^{(L)}\big) +f\big(u^{(R)}\big) -\alpha\big( u^{(R)}-u^{(L)}\big)\right),
\end{equation}
where
\[
\alpha=\max_{u\in\left[ u^{(L)},u^{(R)}\right]}\abs{f'(u)}.
\]
Here we have assumed that $u^{(L)}\leq u^{(R)}$, otherwise we take the maximum over $[u^{(R)},u^{(L)}]$. In practice, we do not solve the maximization problem. We approximate by evaluating $\abs{f'(u)}$ in the points $u^{(L)}$, $u^{(R)}$ and $\tfrac{1}{2}(u^{(L)}+u^{(R)})$ and we take the maximal value. 

Integrals over individual elements in (\ref{DG_Weak}) are evaluated using Gaussian quadrature rules. Basis functions of the space $S_h$ are taken as Legendre polynomials on individual elements, where the support of each basis function is a single element. By writing $u_h$ in terms of basis functions in space and setting the test function $\varphi$ to a elements of the basis, equation (\ref{DG_Weak}) reduces to a system of ordinary differential equations which is solved by the explicit Euler method. We have also implemented higher order Adamsâ€“Bashforth methods, however numerical experiments show that the simple Euler method is sufficiently accurate for our purposes.

The DG method is much less susceptible to the Gibbs phenomenon than the finite element method, however spurious oscillations can still occur locally in the vicinity of discontinuities or steep gradients in the solution. There are several approaches how to treat these local oscillations, e.g. adding local artificial diffusion. In our case, we apply limiters to the DG solution. In our implementation, we use the \emph{modified minmod limiter} from  \cite{Stabilization_original}, cf. also \cite{Stabilization}.

Often the solution of (\ref{DG_Hyperbolic_problem}) is a physical quantity which satisfies some admissibility conditions, e.g. the physical density must be positive. If we obtain a solution which is not in the admissible interval, e.g. due to overshoots or undershoots, the problem can become ill-posed or even undefined. This is our case, since the traffic density $\rho$ must naturally satisfy $\rho\in[0,\rho_{\max}]$. The DG method by itself does not guaranty such bounds are satisfied for the discrete solution. Limiters usually prevent this from happening, however in traffic flows, it is natural that entire regions of the computational domain have $\rho=0$ or $\rho=\rho_{\max}$ and it is easy for the algorithm to produce e.g. negative density due to round-off errors. To prevent this from happening, we use the following procedure. If the average density on an element $K$ is in the admissible interval, we decrease the slope of our solution so that the modified density lies in $[\rho_{\min},\rho_{\max}]$ similarly as in the limiting procedure. The important property is that the integral $\int_{a_K}^{b_K}\rho(x)\ \dd x$ does not change after the application of the limiter. As further insurance, if the average density on an element $K$ is not in the admissible interval $[0,\rho_{\max}]$, then we change the solution such that $\rho\equiv 0$ or $\rho\equiv\rho_{\max}$ on the whole element $K$. The latter case, when the average density on an element is not in the admissible interval is extremely rare and, for us, serves as an indicator that the time step is too large or the mesh is too coarse. Since in this case the described procedure does not conserve the total number of vehicles, we rather decrease the time step or increase the number of elements. Figure \ref{obr_Limiters} demonstrates the effect of applying the minmod limiter along with the described procedures enforcing the admissible interval.

\begin{figure}[t!]\centering
\subfloat[Result without the application of limiters.]{\includegraphics[height=2.4in]{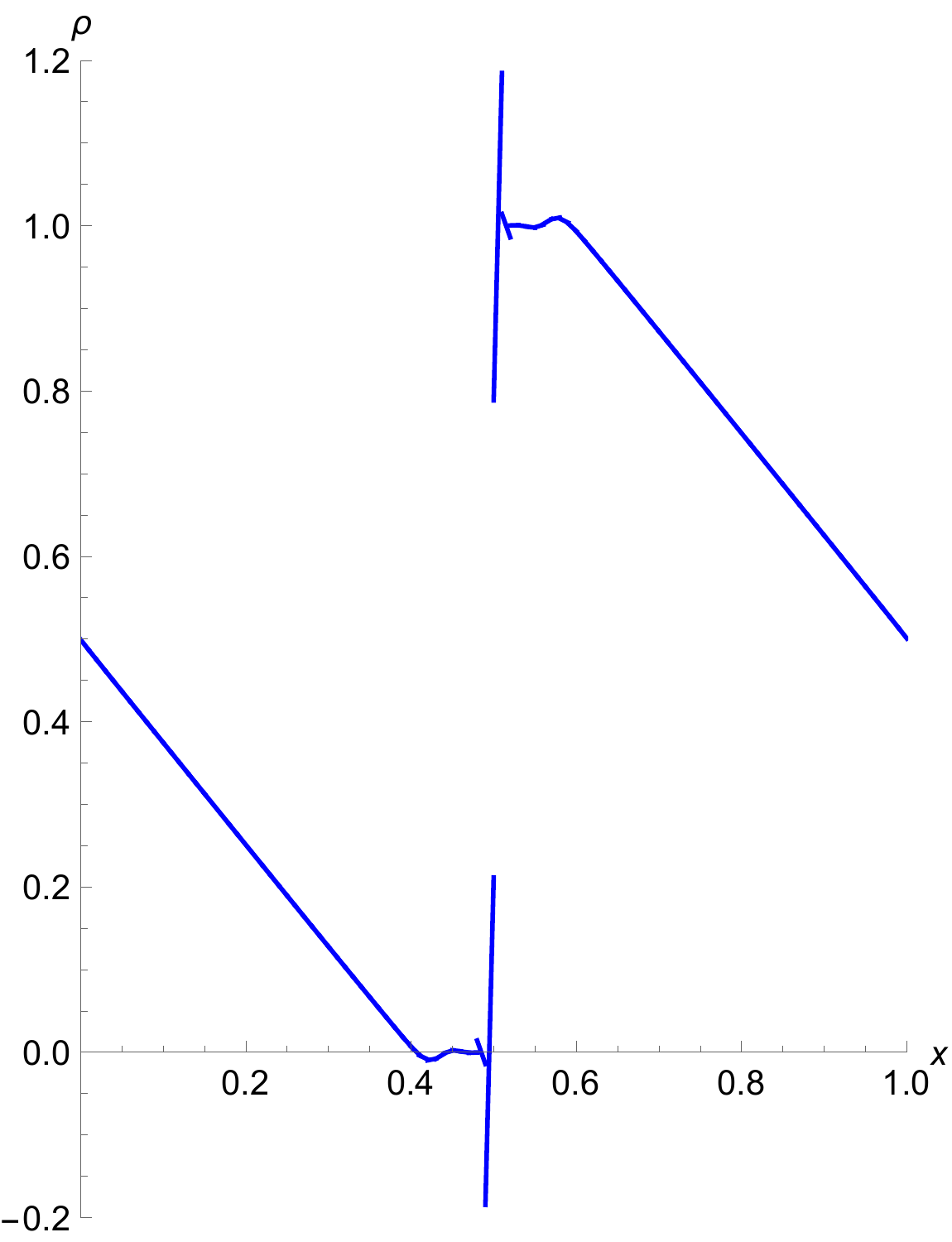}}
\hspace{10pt}
\subfloat[Result with the application of limiters.]{\includegraphics[height=2.4in]{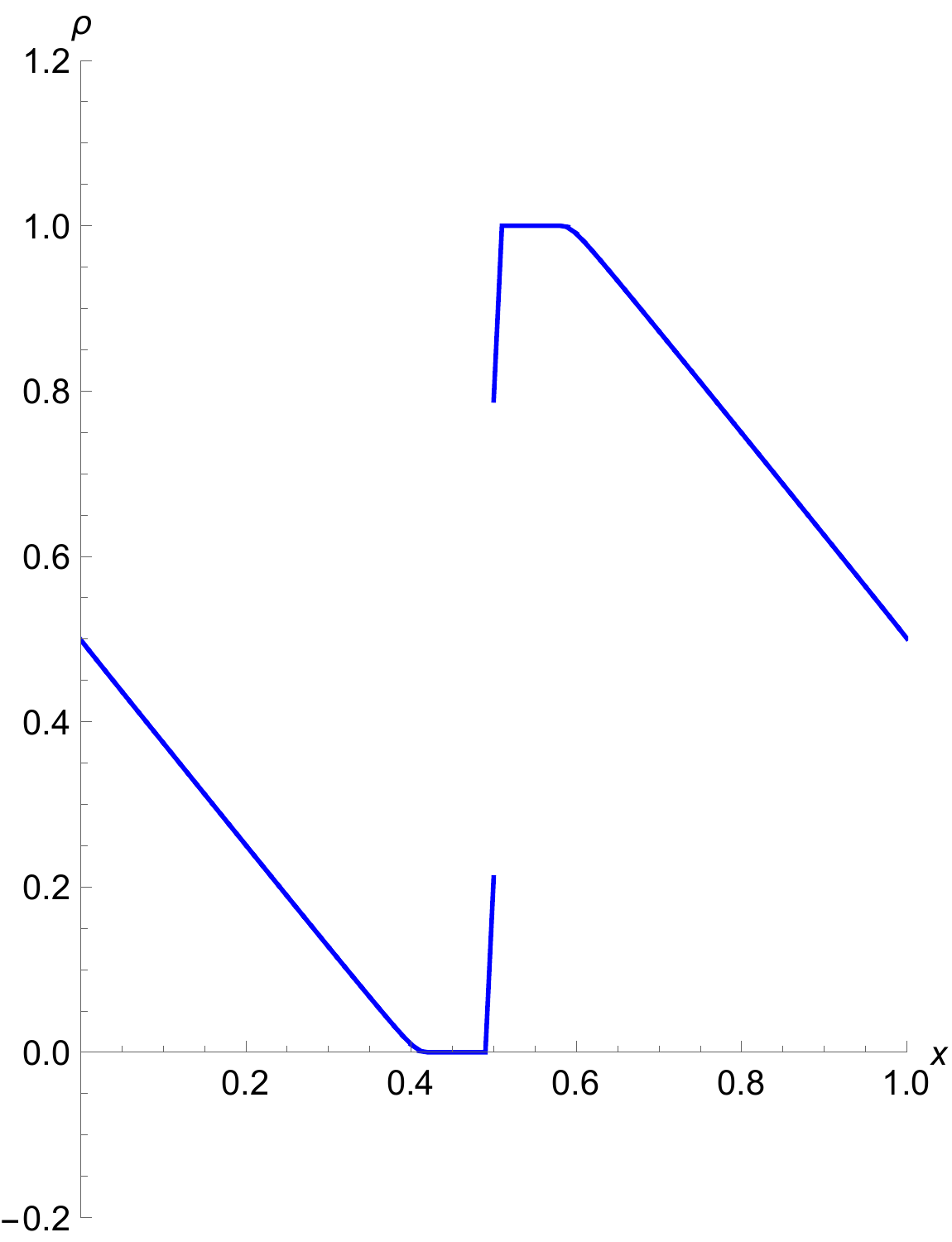}}
\caption{The effect of limiters on the numerical solution.}
\label{obr_Limiters}
\end{figure}

\section{DG method on networks}\label{Sec_DG_method_on_networks}
Now we shall formulate the DG method for LWR models on networks. Throughout this section we shall deal with the simple case of a network with a single junction. This is purely for simplicity which allows us to keep the notation relatively simple. The case of general networks is then a straightforward extension. First we construct suitable numerical fluxes at the junction, then we define the DG scheme on the network using these fluxes. Throughout this section we use the notation from Section \ref{Network}. 

\subsection{Numerical fluxes at junctions}\label{Subsection_Fluxes_junctions}
In order to formulate the DG scheme on a simple network, we first need to construct the numerical fluxes at the junction. We take a different approach from that of \cite{RKDG} and \cite{Networks}. Our approach has the advantage that it is simple and explicitly constructed for all junction types. We shall prove the basic properties of this construction and discuss the differences with the approach of \cite{RKDG} and \cite{Networks}.

At the junction, we consider an incoming road $I_i$ and an outgoing road $I_j$. If these roads were the only roads at the junction, i.e. if they were directly connected to each other, the (numerical) flux of traffic from $I_i$ to $I_j$ would simply be $H\big(\rho_{hi}^{(L)}(b_i,t),\rho_{hj}^{(R)}(a_j,t)\big)$, where $\rho_{hi}$ and $\rho_{hj}$ are the DG solutions on $I_i$ and $I_j$, respectively. From the traffic distribution matrix, we know the ratios of the traffic flow distribution to the outgoing roads. Thus, we take the numerical flux $H_j(t)$ at the left point of the outgoing road $I_j$, i.e. at the junction, at time $t$ as
\begin{equation}\label{num_flux_out}
H_j(t):=\sum_{i=1}^{n} \alpha_{j,i} H\big(\rho_{hi}^{(L)}(b_i,t),\rho_{hj}^{(R)}(a_j,t)\big),
\end{equation}
for $j=n+1,\ldots,n+m$. The numerical flux $H_j(t)$ can be viewed as the DG analogue of taking the combined traffic outflow $\sum_{i=1}^{n} \alpha_{j,i} Q_e\big(\rho_i^{(L)}(b_i,t)\big)$ from all incoming roads and prescribing it as the inflow of traffic to the road $I_j$.

Similarly, we take the numerical flux $H_i(t)$ at the right point of the incoming road $I_i$, i.e. at the junction, at time $t$ as
\begin{equation}\label{num_flux_in}
H_i(t):=\sum_{j=n+1}^{n+m} \alpha_{j,i} H\big(\rho_{hi}^{(L)}(b_i,t),\rho_{hj}^{(R)}(a_j,t)\big),
\end{equation}
for $i=1,\ldots,n$. Again, this can be viewed as an approximation of the traffic  flow $\sum_{j=n+1}^{n+m} \alpha_{j,i} Q_e\big(\rho_j^{(R)}(a_j,t)\big)$ being prescribed as the outflow of traffic from $I_i$.

This choice of numerical fluxes at the junction satisfies the discrete analogue to the Rankine--Hugoniot condition (\ref{Junction_equation}), which in turn means that the DG solution using these fluxes conserves the total amount of cars passing through the junction (cf. Theorem \ref{veta_properitoes_of_solution3}). 
\begin{lemma}[Discrete Rankine--Hugoniot condition]\label{veta_properitoes_of_solution}
The numerical fluxes (\ref{num_flux_out}) and (\ref{num_flux_in}) satisfy the discrete version of the Rankine--Hugoniot condition (\ref{Junction_equation}):
\begin{equation}
\sum_{i=1}^n H_i(t)=\sum_{j=n+1}^{n+m} H_j(t).
\end{equation}
\end{lemma}
\begin{proof}
From the definition of $H_i$ and $H_j$, we immediately obtain
\begin{align*}
\sum_{i=1}^n H_i(t)=\sum_{i=1}^n\sum_{j=n+1}^{n+m} \alpha_{j,i} H\big(\rho_{hi}^{(L)}(b_i,t),\rho_{hj}^{(R)}(a_j,t)\big)
=\sum_{j=n+1}^{n+m}\sum_{i=1}^n \alpha_{j,i} H\big(\rho_{hi}^{(L)}(b_i,t),\rho_{hj}^{(R)}(a_j,t)\big)=\sum_{j=n+1}^{n+m} H_j(t).
\end{align*}
\end{proof}

\subsection{DG method on networks}\label{sec_DG_Networks}
Now we can formulate the DG method for the simplified network with one junction using the numerical fluxes defined in (\ref{num_flux_out}) and (\ref{num_flux_in}). Then the case of general networks is a straightforward generalization, where the aforementioned construction of numerical fluxes at junctions is applied on each junction separately.

We consider the DG formulation (\ref{DG_Weak}) on every incoming and outgoing road represented by the intervals $(a_i,b_i), i=1,\ldots,n$ and $(a_j,b_j), j=n+1,\ldots,n+m$, respectively. Since the DG method is applied on finite intervals, we replace the endpoints at $\pm\infty$ from Definition \ref{Def_network} by artificial inflow/outflow boundaries at finite points along with inflow Dirichlet data. For every interval $(a_k,b_k), k=1,\ldots,n+m$, we consider a partition $\kT_{hk}$ along with the corresponding discrete space $S_{hk}$. We write the DG formulation directly for the case of LWR models (\ref{LWR_problem}) with unknown density $\rho$ and flux $Q_e(\rho)$.

\begin{definition}[DG formulation on a simple network]\label{def_DG_solution_network}
We seek functions $\rho_{hk}\in\kC^1\left(\left[ 0,T\right];S_{hk}\right)$, $k=	1,\ldots,n+m$ satisfying the following.
\begin{itemize}
\item[(i)] \emph{Incoming roads:} For all $i=1\ldots,n$ and all $\varphi_i\in S_{h_i}$
\begin{equation}\label{DG_Weak_inc}
\begin{split}
\int_{a_i}^{b_i} (\rho_{hi})_t\varphi_i\ \dd x-\sum_{K\in\kT_{hi}}\int_K Q_e(\rho_{hi})\varphi_i'\ \dd x+ \sum_{x\in\kF_{hi}^I}H\big(\rho_{hi}^{(L)},\rho_{hi}^{(R)}\big)\left[\varphi_i\right]&\\ +H_i\varphi_i^{(L)}(b_i) -H\big(\rho_{Di},\rho_{hi}^{(R)}(a_i)\big)\varphi_i^{(R)}(a_i)&=0,
\end{split}
\end{equation}
where $H_i=H_i(t)$ is the numerical flux defined in (\ref{num_flux_in}) and $\rho_{Di}$ is the Dirichlet datum corresponding to the left artificial inflow boundary point $a_i$ of $(a_i,b_i)$.
\item[(ii)] \emph{Outgoing roads:} For all $j=n+1,\ldots,n+m$ and all $\varphi_j\in S_{h_j}$
\begin{equation}\label{DG_Weak_out}
\begin{split}
\int_{a_j}^{b_j} (\rho_{hj})_t\varphi_j\ \dd x-&\sum_{K\in\kT_{hj}}\int_K Q_e(\rho_{hj})\varphi_j'\ \dd x+ \sum_{x\in\kF_{hj}^I}H\big(\rho_{hj}^{(L)},\rho_{hj}^{(R)}\big)\left[\varphi_j\right]\\  &\quad +H\big(\rho_{hj}^{(L)}(b_j),\rho_{hj}^{(L)}(b_j)\big)\varphi_j^{(L)}(b_j) -H_j\varphi_j^{(R)}(a_j)=0,
\end{split}
\end{equation}
where $H_j=H_j(t)$ is the numerical flux defined in (\ref{num_flux_out}).
\end{itemize}
\end{definition}

\begin{remark}
We note that the choice of the arguments in the numerical flux at the artificial boundary point $b_j$ in (\ref{DG_Weak_out}) corresponds to an outflow boundary condition. This term could be rewritten using the original physical flux $Q_e(\rho_{hj}^{(L)}(b_j))$ due to consistency of the numerical flux $H$.
\end{remark}

As a corollary of Lemma \ref{veta_properitoes_of_solution}, we get the conservation of the total number of cars in the network in the DG solution up to the contribution of the inflow and outflow artificial boundaries.

\begin{theorem}[Conservation property of the DG scheme]\label{veta_properitoes_of_solution3}
The DG scheme from Definition \ref{def_DG_solution_network} conserves the total number of vehicles in the network in the sense that
\begin{equation}
\frac{\dd}{\dd t}\sum_{k=1}^{n+m}\int_{a_k}^{b_k}\rho_{hk}\,\dd x= \sum_{i=1}^n H\big(\rho_{Di},\rho_{hi}^{(R)}(a_i)\big) -\sum_{j=n+1}^{n+m} Q_e\big(\rho_{hj}^{(L)}(b_j)\big).
\nonumber
\end{equation}
\end{theorem}
\begin{proof}
We set all test functions $\varphi_k\equiv 1$ for all $k=1,\ldots,n+m$ and sum together all of the equations (\ref{DG_Weak_inc}) and (\ref{DG_Weak_out}) for all $i$ and $j$. We get
\begin{equation}
\begin{split}
\frac{\dd}{\dd t}\sum_{k=1}^{n+m}\int_{a_k}^{b_k}\rho_{hk}\,\dd x +\sum_{i=1}^n H_i -\sum_{j=n+1}^{n+m} H_j+\sum_{j=n+1}^{n+m} Q_e\big(\rho_{hj}^{(L)}(b_j)\big)  -\sum_{i=1}^n H\big(\rho_{Di},\rho_{hi}^{(R)}(a_i)\big) =0.
\end{split}
\nonumber
\end{equation}
The second and third terms cancel one another since $\sum_i H_i-\sum_j H_j=0$ due to Lemma \ref{veta_properitoes_of_solution}. This completes the proof.
\end{proof}

We note that although the numerical fluxes defined in (\ref{num_flux_out}) and (\ref{num_flux_in}) are defined using the traffic distribution matrix, the resulting fluxes $H_i, H_j$ do not satisfy the traffic distribution condition (ii) from Definition \ref{def_network_solution} exactly, but with an error given in the following lemma.

\begin{theorem}[Traffic distribution error]\label{veta_properitoes_of_solution2}
The numerical fluxes (\ref{num_flux_out}) and (\ref{num_flux_in}) satisfy
\begin{equation}
H_j(t)=\sum_{i=1}^n \alpha_{j,i} H_i(t) +E_j(t)
\end{equation}
for all $j=n+1,\ldots,n+m$, where the error term is
\begin{equation}\label{Traffic_distribution error_1}
E_j(t)=\sum_{i=1}^n\sum_{\substack{l=n+1\\l\neq j}}^{n+m} \alpha_{j,i}\alpha_{l,i} \big(H_{i,j}(t) -H_{i,l}(t)\big),
\end{equation}
where $H_{i,j}(t):=H\big(\rho_{hi}^{(L)}(b_i,t),\rho_{hj}^{(R)}(a_j,t)\big)$.
\end{theorem}
\begin{proof}
By definition (\ref{num_flux_out}),
\begin{equation}
H_j(t)=\sum_{i=1}^n\alpha_{j,i}H_{i,j}(t) =\sum_{i=1}^n\alpha_{j,i}H_{i}(t) +\underbrace{\sum_{i=1}^n\alpha_{j,i}\big(H_{i,j}(t)-H_i(t)\big)}_{E_j(t)},
\end{equation}
where $E_j(t)$ is the error term which we will show has the form (\ref{Traffic_distribution error_1}): by definition (\ref{num_flux_in}), we have
\begin{equation}
\begin{split}
E_j(t)&= \sum_{i=1}^n\alpha_{j,i}\Big(H_{i,j}(t)-\sum_{l=n+1}^{n+m}\alpha_{l,i}H_{i,l}(t)\Big)= \sum_{i=1}^n\alpha_{j,i}\sum_{l=n+1}^{n+m}\alpha_{l,i}\big(H_{i,j}(t)-H_{i,l}(t)\big)
\\ &=\sum_{i=1}^n\sum_{\substack{l=n+1\\l\neq j}}^{n+m} \alpha_{j,i}\alpha_{l,i} \big(H_{i,j}(t) -H_{i,l}(t)\big),
\nonumber
\end{split}
\end{equation}
since $\sum_{l=n+1}^{n+m}\alpha_{l,i}=1$ due to (\ref{sum_alpha}). This completes the proof.
\end{proof}

\begin{example}
Let us consider a junction with one incoming and two outgoing roads. Assume for example that $\rho_{h1}^{(L)}(b_1,0)=0.5$, $\rho_{h2}^{(R)}(a_2,t)=0.2$, $\rho_{h3}^{(R)}(a_3,t)=0$, $\alpha_{2,1}=0.75$ and $\alpha_{3,1}=0.25$. We use the Greenshields model (with $v_{\max}=\rho_{\max}=1$) and the Lax--Friedrichs flux (\ref{Lax_Friedrichs_flux}). Then
\[
H_2(0)=\alpha_{2,1} H\big(\rho_{h1}^{(L)}(b_1,0),\rho_{h2}^{(R)}(a_2,0)\big)=0.22125
\]
and
\[
H_1(0)=\alpha_{2,1} H\big(\rho_{h1}^{(L)}(b_1,0),\rho_{h2}^{(R)}(a_2,0)\big)+\alpha_{3,1} H\big(\rho_{h1}^{(L)}(b_1,0),\rho_{h3}^{(R)}(a_3,0)\big)=0.315.
\]
Since $H_2(0)=0.2212\neq 0.23625=\alpha_{2,1}H_1(0)$, we see that in this case $H_2(\cdot)\neq\alpha_{2,1}H_1(\cdot)$, thus the property (ii) in Definition \ref{def_network_solution} is not satisfied exactly but with a small relative error of approximately $6\%$.	
\end{example}

Here we would like to comment on the interpretation of Theorem \ref{veta_properitoes_of_solution2} and on the differences between our construction of the numerical fluxes and that of \cite{Networks}, \cite{RKDG}.

\begin{enumerate}
\item \emph{Dedicated turning lanes.} Consider as an example a junction with one incoming and two outgoing roads. The flux considered in \cite{Networks}, \cite{RKDG} which satisfies Definition \ref{def_network_solution} is constructed as follows, cf. \cite{RKDG}. We compare the maximum possible fluxes which can inflow into the junction from the incoming road ($\gamma_1^{\max}$) or outflow from the junction to the outgoing roads ($\gamma_2^{\max}$ and $\gamma_3^{\max}$). We take $\gamma=\min\lbrace\gamma_1^{\max},\frac{\gamma_2^{\max}}{\alpha_{2,1}},\frac{\gamma_3^{\max}}{\alpha_{3,1}}\rbrace$ and use it as inflow into the junction from the incoming road, i.e. $\widehat{H}_1=\gamma$. The outflow to outgoing roads is then $\widehat{H}_2=\alpha_{2,1}\gamma$ and $\widehat{H}_3=\alpha_{3,1}\gamma$. Consider a situation, when e.g. the left outgoing road is blocked by a traffic jam. Then $\gamma=\widehat{H}_1=\widehat{H}_2=\widehat{H}_3=0$ and the whole junction is blocked. Namely, the cars cannot go turn right either, even though the right road might be completely empty. This corresponds to the situation where the roads are single-lane, thus the cars which want to turn right are blocked by the left-going cars which cannot proceed due to the congestion in the left outgoing road.

In our approach, we calculate the simple numerical fluxes $H$, where the left value is the traffic density of an incoming road and the right value is the traffic density of one of the outgoing roads. Then we take the possible fluxes and multiply them by the traffic distribution coefficients. If we consider the traffic jam in the left outgoing road as above, the flow into this road will be zero, however, the cars can still go into the second outgoing road according to the traffic distribution coefficient. Thus there will be a nonzero flow of traffic into the right lane. This can be interpreted as the existence of dedicated turning lanes in the road, so the left-going traffic, which is standing still, does not block the junction for the right-going traffic. 

Since macroscopic models in general are used to model long (multi--line) roads with huge numbers of vehicles, we view the behavior of our model as more realistic in this situation. The original approach from Definition \ref{def_network_solution} is aimed for single--lane roads, where passing is not possible.

\item \emph{Flexibility of drivers' preferences.} When inspecting the traffic distribution error (\ref{Traffic_distribution error_1}), one may ask when is $E_j$ equal to zero, i.e. when is the traffic distributed \emph{exactly} according to the a priori preferences given by the matrix $A$. This happens (among other), when for every incoming road $I_i$, all the (numerical) fluxes to the outgoing roads are equal (i.e. $H_{i,j}$=$H_{i,l}$ for all $l=n+1,\ldots,n+m$). From the point of view of a driver on road $I_i$ approaching the junction: the driver evaluates how all the traffic from $I_i$ would flow to each outgoing road $I_j$ individually (this is $H_{i,j}$). If all these flows to the outgoing roads are equal, the driver behaves according to the predetermined preferences given by the coefficients $\alpha_{j,i}$ This is however the idealized situation. When the flows to outgoing roads are not equal, there is some imbalance in the traffic network and the driver might decide to change his preference on the spot. 

Consider again the situation, when the left outgoing road is congested while the right road is empty. It is then natural that some drivers decide to change their original preference and take an alternative route, turning right instead of left. This is natural, since for most destinations, there are several possible routes and the driver can adapt his course according to the current situation. This is especially the case for city traffic. Thus we can interpret our approach to the flows at junctions as a certain flexibility of the drivers' preferences, while in the approach from Definition \ref{def_network_solution} the predetermined traffic distribution is strictly adhered to.

\item \emph{Traffic lights.} Traffic lights are considered in \cite{Networks}, where an example with two incoming and two outgoing roads is used. It is not explicitly stated, but only full green lights are allowed in this case. The presence of green or red light is then determined by the traffic distribution matrix. For example, If road $I_1$ has a green light and road $I_2$ has a red light, then $\sum_{j=n+1}^{n+m}\alpha_{j,1}=1$ and $\sum_{j=n+1}^{n+m}\alpha_{j,2}=0$. If we were to prescribe green only for some outgoing roads (not full green) for road $I_1$, i.e. $\sum_{j=n+1}^{n+m}\alpha_{j,1}<1$, the distribution of traffic will then not have to be satisfied.

As we mention above, our approach can be interpreted as describing dedicated turning lanes. We can therefore implement arbitrary time-varying traffic light combinations simply by setting certain coefficients to zero (red light). For example, if there is red for the direction from incoming road $I_i$ to outgoing road $I_j$ at the time $t$, then we set $\alpha_{j,i}(t):=0$. For the directions with a green light at time $t$, the traffic distribution coefficients are simply taken as the predetermined coefficients, i.e. $\alpha_{j,i}(t):=\alpha_{j,i}$. This allows us to simulate a wide range of traffic light combinations from the real world. On the other hand, traffic distribution error may increase or decrease depending on the length of the green light interval for each directions. We may again ask when will the traffic be distributed exactly according to the a priori preferences given by matrix $A$, as in Theorem \ref{veta_properitoes_of_solution2}. This happens (among other) in the idealized situation from the previous point (2) with full green lights. Under different circumstances (congestion in one outgoing road) and not full green lights, it is again natural that some real-world drivers would change their original preference, taking an alternative route, choosing a direction with longer green light intervals and/or higher traffic flow.

In Section \ref{num_exp_traffic_lights}, we demonstrate the performance of our method for a junction with 4 incoming and 4 outgoing roads and a complex periodic traffic light pattern with three distinct phases taken from a real-world junction.

\item \emph{Discontinuous Galerkin setting.} Finally, we can view our approach in the context of the DG method and its philosophy. Consider, for example, how the DG method treats Dirichlet boundary conditions. These are not strictly enforced as exact boundary values of the discrete solution as in the finite element method. Instead they are enforced in some weak sense, via numerical fluxes on the boundary (first order hyperbolic problems) or by penalization (elliptic problems). The result is that the Dirichlet boundary conditions are satisfied not exactly, but with some smaller or larger error, depending on the situation. The same holds for global continuity of the discrete solution, which is not enforced strictly, but rather in some weak sense (again via numerical fluxes or penalization terms). This flexibility in enforcing certain conditions is one of the main aspects that gives the DG method its robustness as opposed to the finite element method in situations such as steep boundary layers or discontinuities in the solution. We therefore view our approach to the fluxes at junctions as natural in the DG setting, where the boundary conditions at the junction (i.e. the traffic distribution coefficients) are satisfied exactly only in ideal circumstances (cf. the previous point (2) above), but with some smaller or larger error otherwise.
\end{enumerate}

\begin{remark}
Taking the above considerations into account, it seems that in our approach a more appropriate name for the traffic distribution matrix would be \emph{traffic preference matrix}, since it might not be realized exactly in extreme traffic situations.
\end{remark}

\section{Numerical results}\label{sec_Numerical_results}
In this section we present numerical results obtained using the method described in Sections \ref{DG_intro} and \ref{Sec_DG_method_on_networks}. We use $P^1$ elements with two quadrature points in each element. The implementation was done in the C++ language. We show the result of calculation of a bottleneck and on networks. We also compare our results with the approach from \cite{RKDG} where the authors use the maximum possible fluxes from Definition \ref{def_network_solution}.

\subsection{Bottleneck}
First we demonstrate results for a single road with a bottleneck, cf. Figure \ref{obr_Bottleneck}. The parameters taken in this example are taken from typical construction bottlenecks on highways in the Czech Republic. In Sector 1 and 4, we have maximal velocity $v_{\max,1}=1.3$ (corresponding to the speed limit 130 km/h) and maximal density $\rho_{\max,1}=2$, which corresponds to two lanes. The length of Sector 1 is $L_1=2$ (i.e. 2 km) and the length of Sector 4 is $L_4=1$. Sector 2 is a short sector with length $L_2=0.5$ and with decreased maximal velocity $v_{\max,2}=1$ (i.e. 100 km/h) and maximal density $\rho_{\max,2}=2$. Sector 3 is the bottleneck, where the maximal density is $\rho_{\max,3}=1$, which corresponds to one lane. The maximal velocity is $v_{\max,3}=0.8$ (i.e. 80 km/h) and the length of this sector is $L_3=2$.

The cars go from left to right. The boundary condition on the left is $\rho\left(0,t \right) = \tfrac{1}{20}\sin\left( \tfrac{2\pi t}{7}-\tfrac{\pi}{2}\right)+0.18$ to simulate time--varying traffic. The initial condition is an empty road. We use the Greenshields model. The time--step size is $\tau=10^{-4}$ and the length of each element is $h=\tfrac{1}{150}$.

In Figure \ref{obr_Bottleneck_result} we can observe the emergence of a traffic congestion between Sector 2 and Sector 3. The traffic congestion spreads backwards to Sector 1 and becomes longer or shorter depending on the boundary influx. Because $\rho(x,t)<\rho_{\max,i}$ for all $x$, $t$ and all sectors, the cars in the traffic congestion are still moving.

We note the relationship between maximal velocity and traffic density depending on the presence of a traffic congestion in Sectors 1 and 2. Without any traffic congestion, the density in Sector 1 (with higher maximal velocity) is lower than the density in Sector 2, cf. Figure \ref{obr_Bottleneck_result_3}. Conversely, with the traffic congestion, the density in Sector 1 is higher than the density in Sector 2, cf. Figure \ref{obr_Bottleneck_result_19}. This behavior arises since the traffic flow is the same in both sectors.
\begin{figure}[h!]\centering
\includegraphics[height=1.3in]{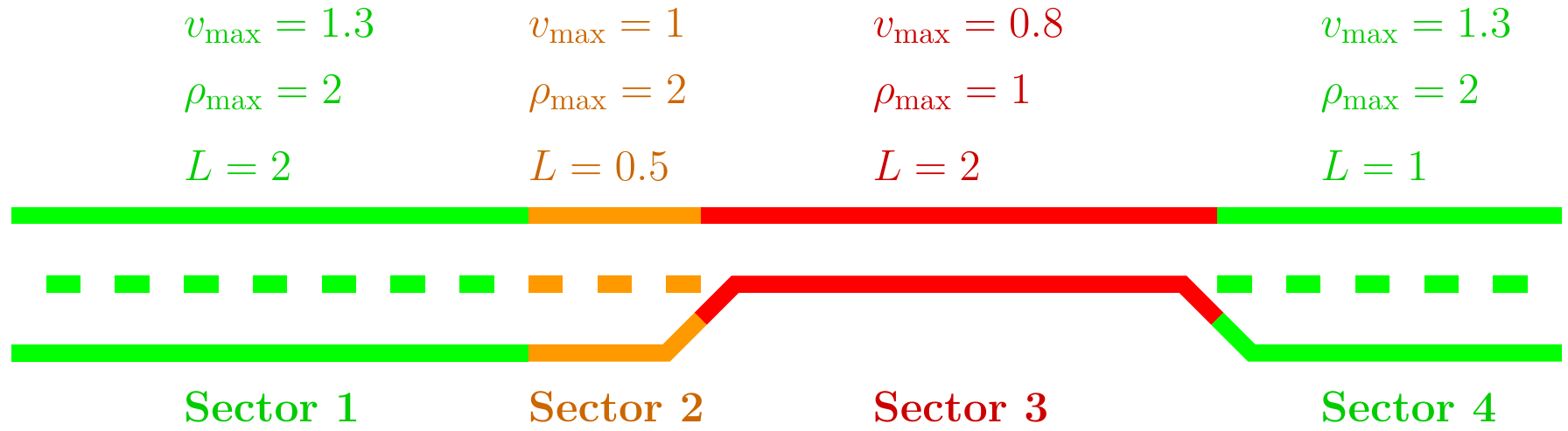}
\caption{Test road with bottleneck.}
\label{obr_Bottleneck}
\end{figure}
\begin{figure}[t!]\centering
\subfloat[$t=3$.]{\label{obr_Bottleneck_result_3}
\includegraphics[height=1.59in]{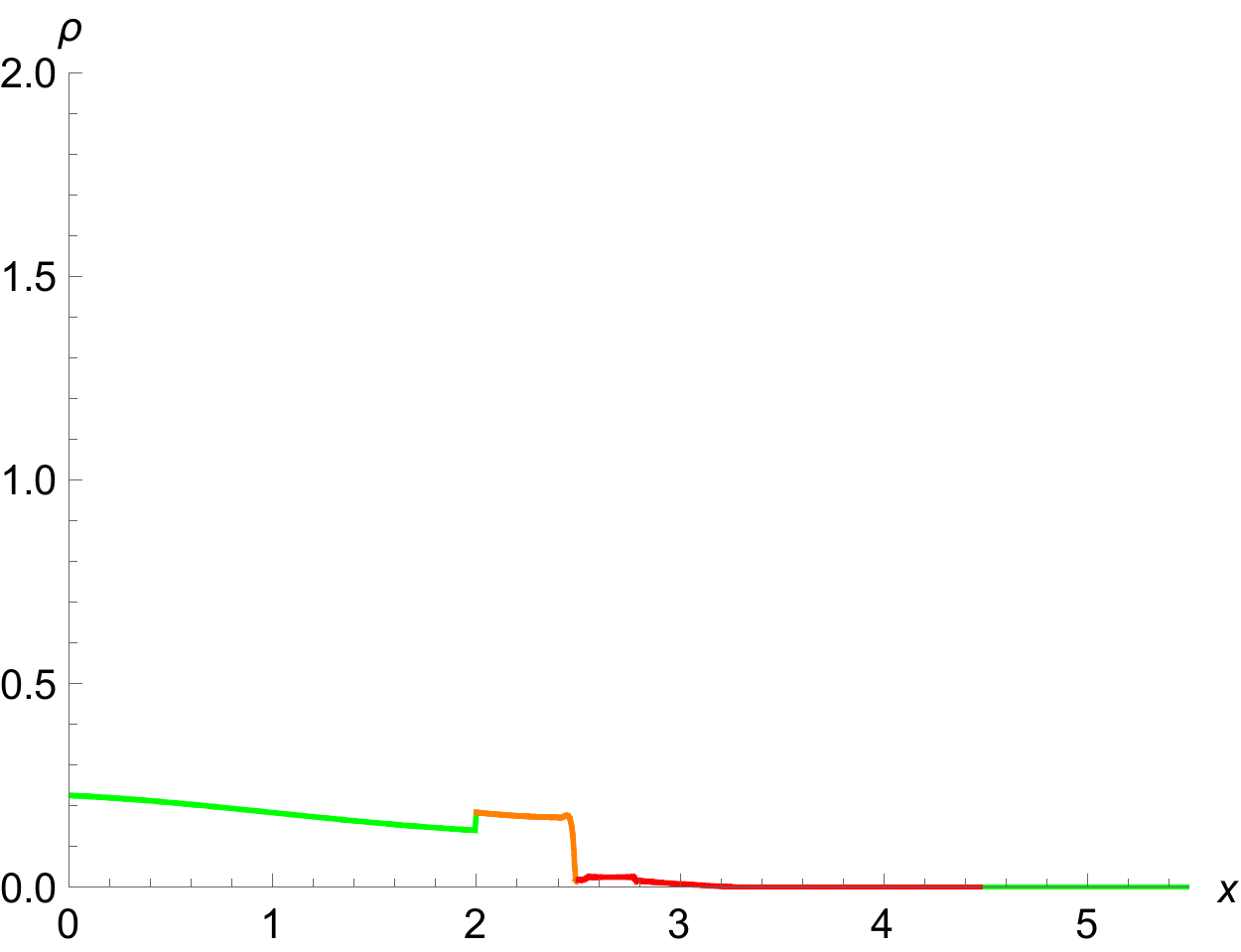}}
\hspace{5pt}
\subfloat[$t=5$.]{\includegraphics[height=1.59in]{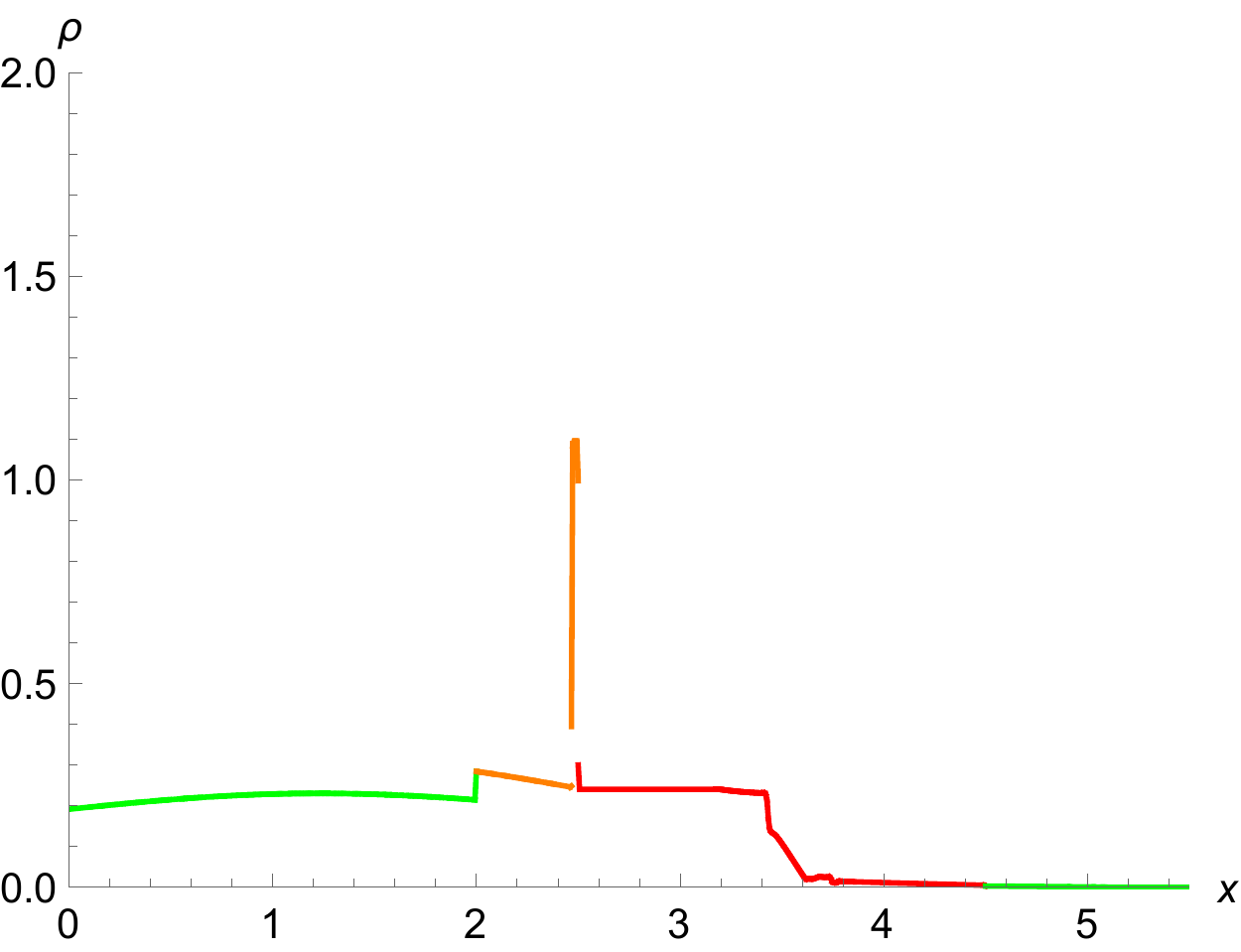}}
\hspace{5pt}
\subfloat[$t=7$.]{\includegraphics[height=1.59in]{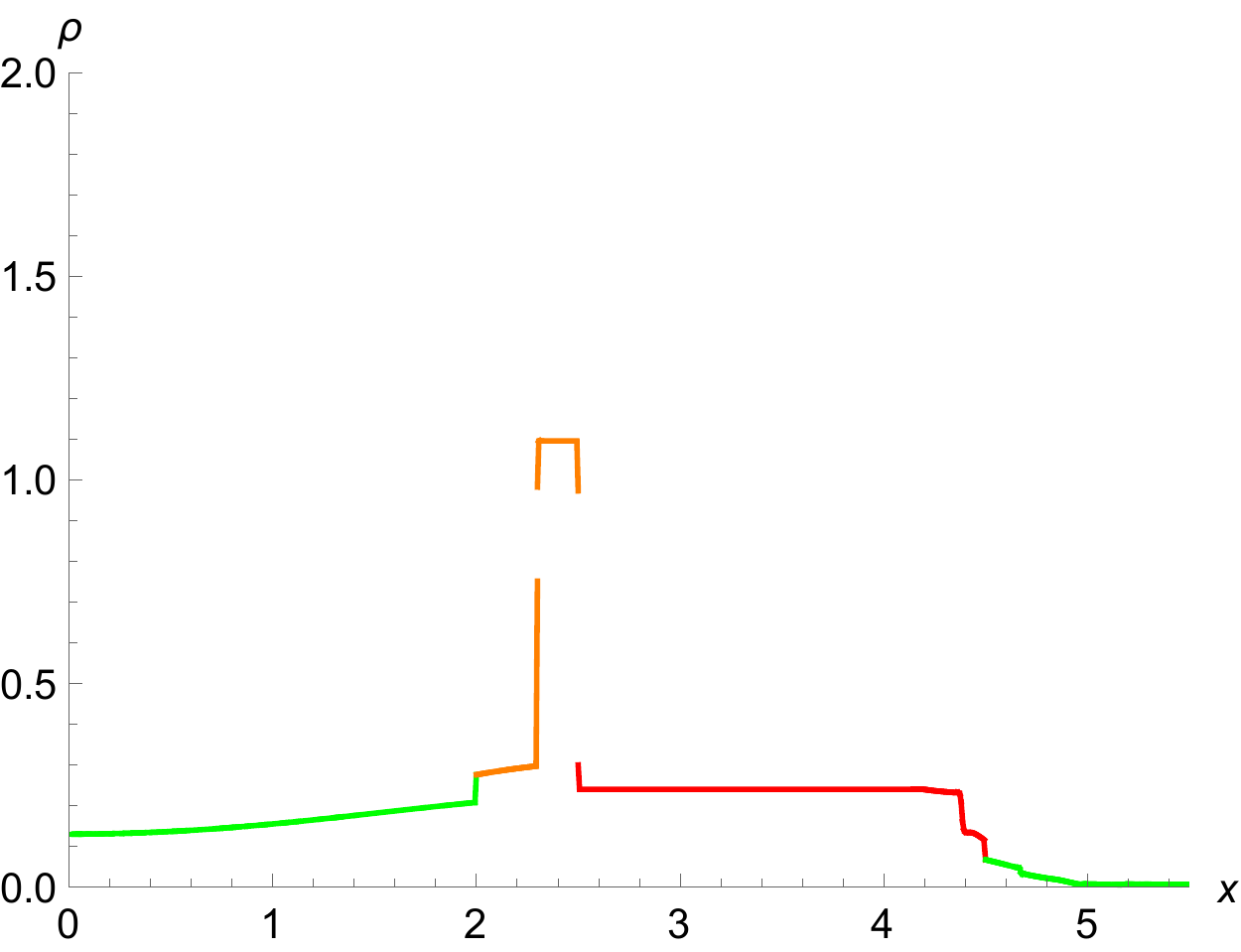}}
\hspace{5pt}
\subfloat[$t=9$.]{\includegraphics[height=1.59in]{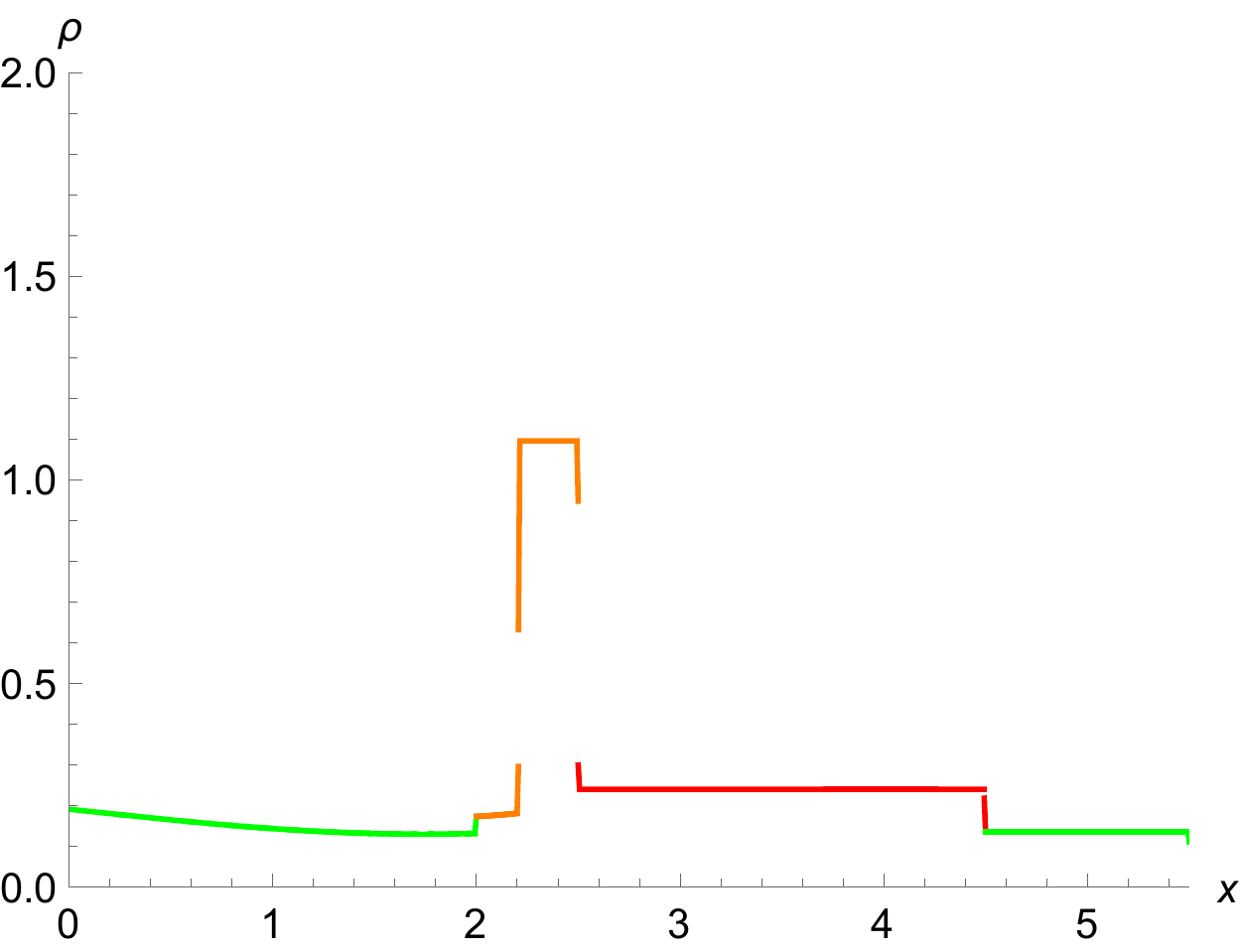}}
\hspace{5pt}
\subfloat[$t=11$.]{\includegraphics[height=1.59in]{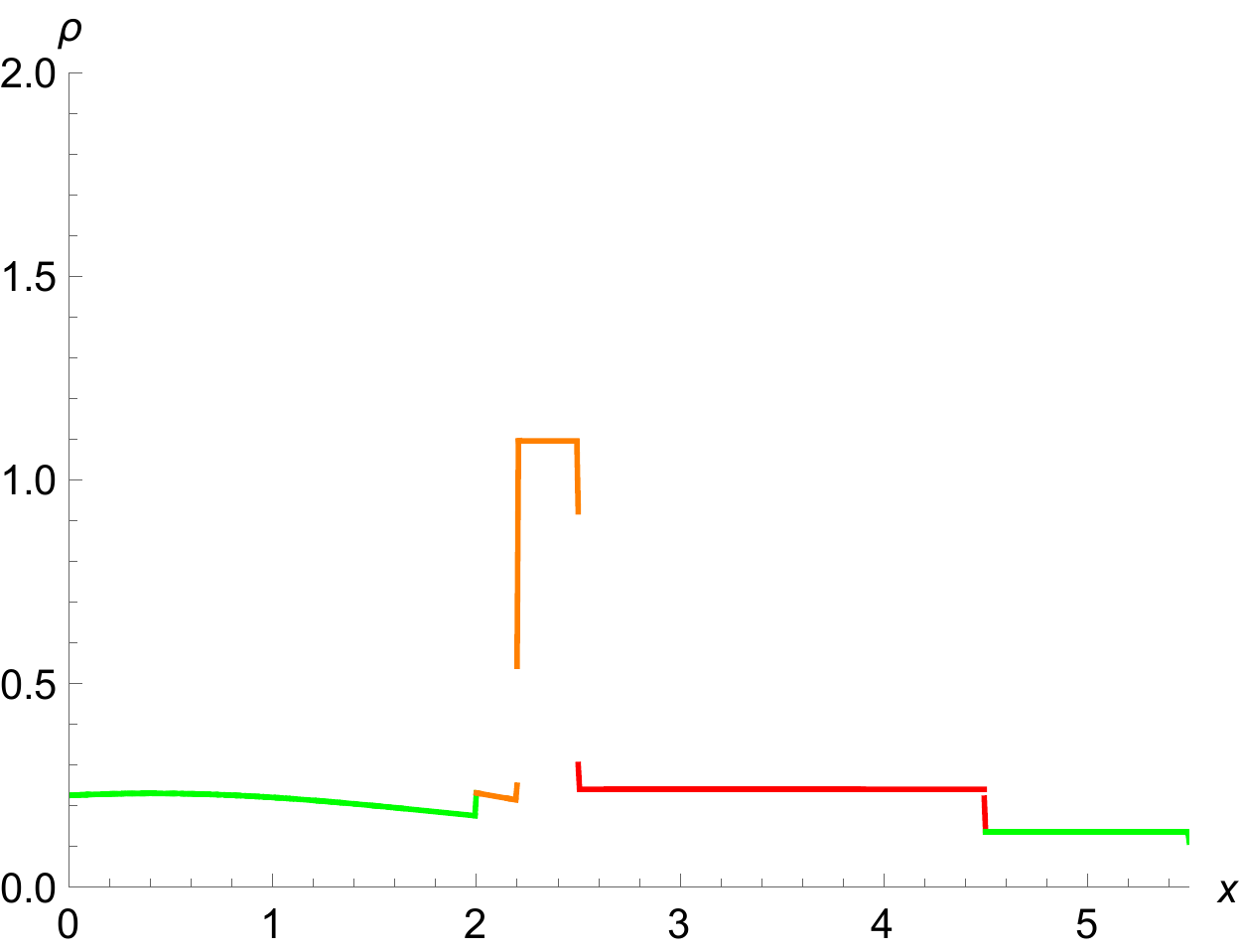}}
\hspace{5pt}
\subfloat[$t=13$.]{\includegraphics[height=1.59in]{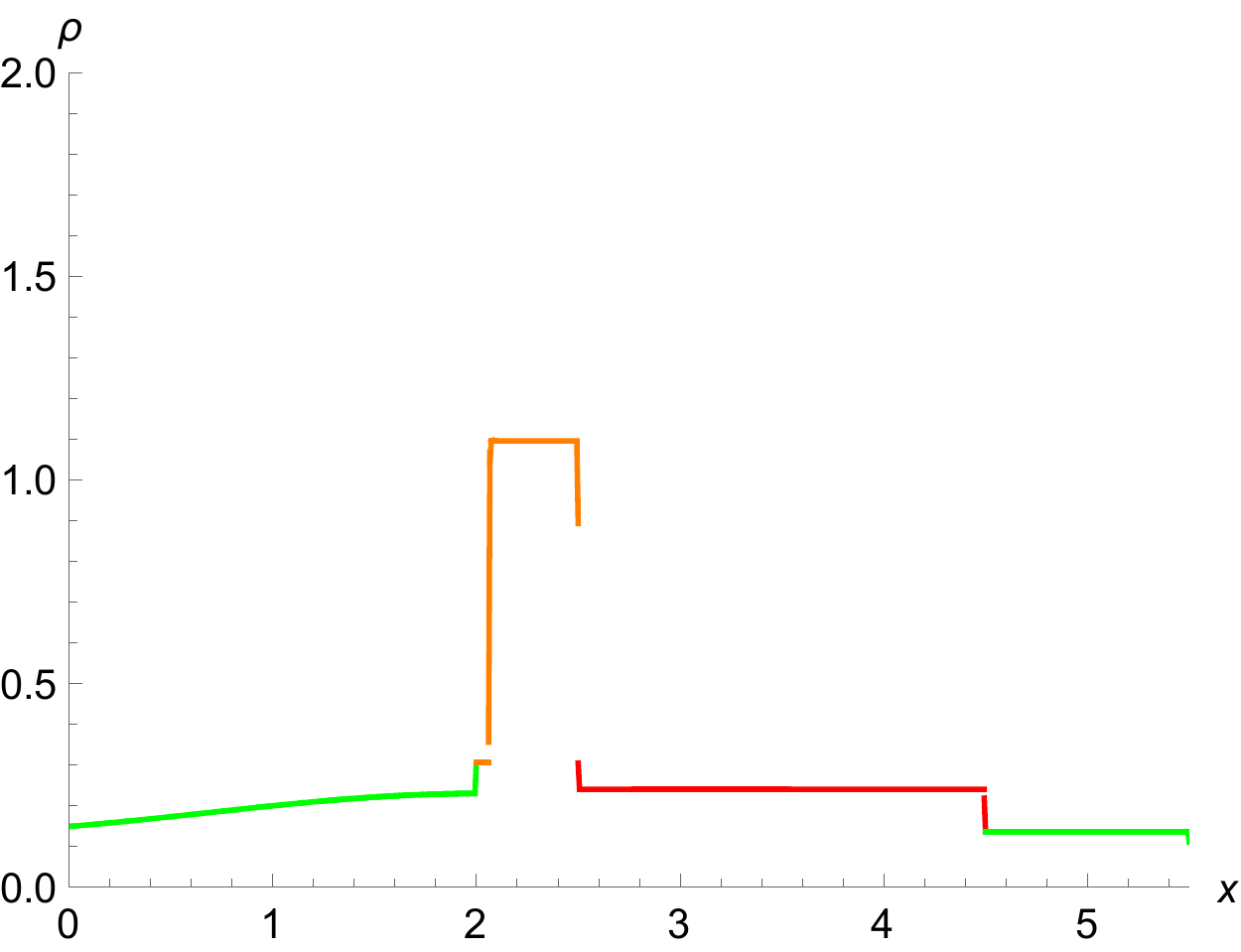}}
\hspace{5pt}
\subfloat[$t=15$.]{\includegraphics[height=1.59in]{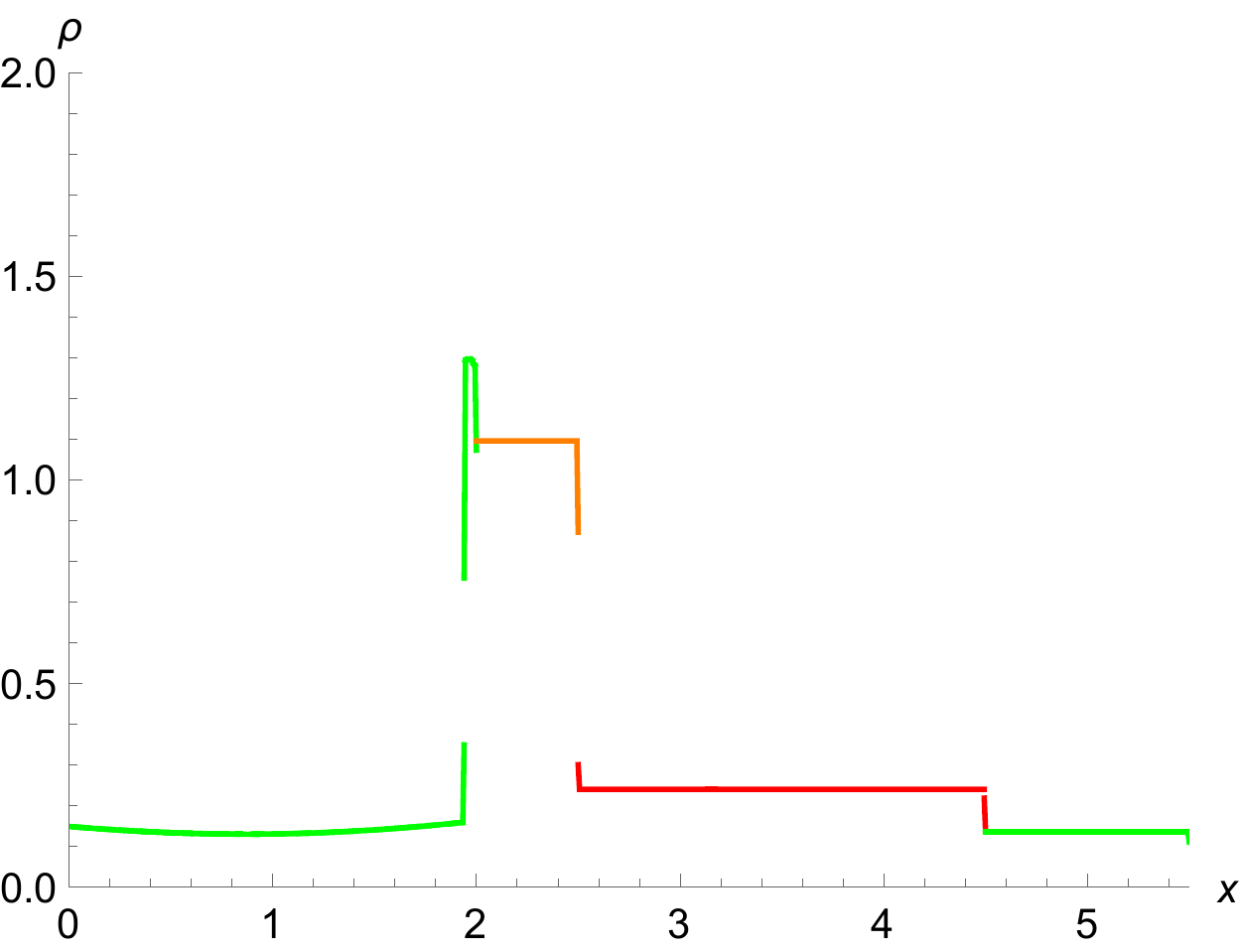}}
\hspace{5pt}
\subfloat[$t=17$.]{\includegraphics[height=1.59in]{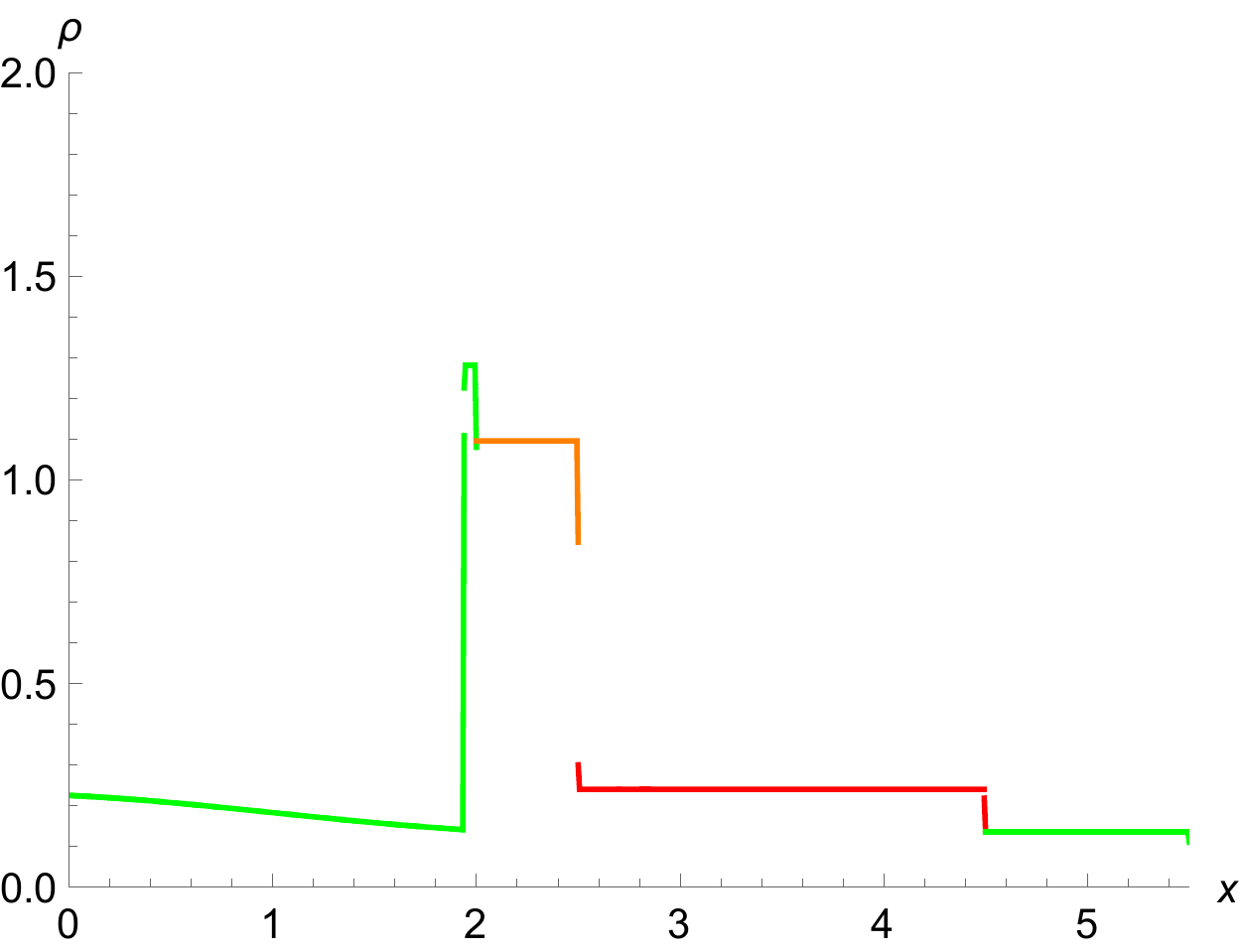}}
\hspace{5pt}
\subfloat[$t=19$.]{\label{obr_Bottleneck_result_19}\includegraphics[height=1.59in]{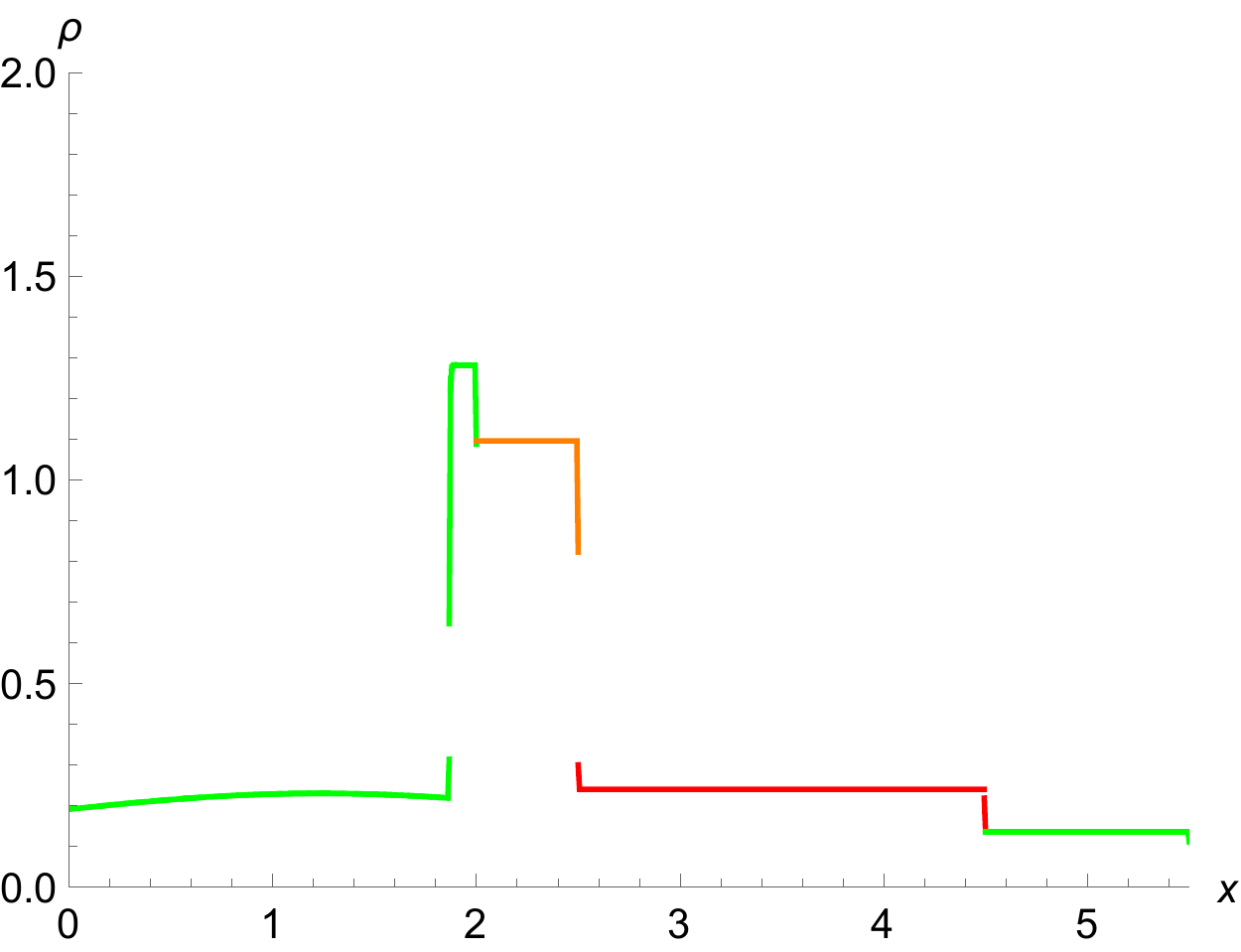}}
\caption{Bottleneck -- density on \textcolor{green}{Sector 1}, \textcolor{orange}{Sector 2}, \textcolor{red}{Sector 3} and \textcolor{green}{Sector 4}.}
\label{obr_Bottleneck_result}
\end{figure}

\begin{figure}[h]\centering
\includegraphics[height=1.78in]{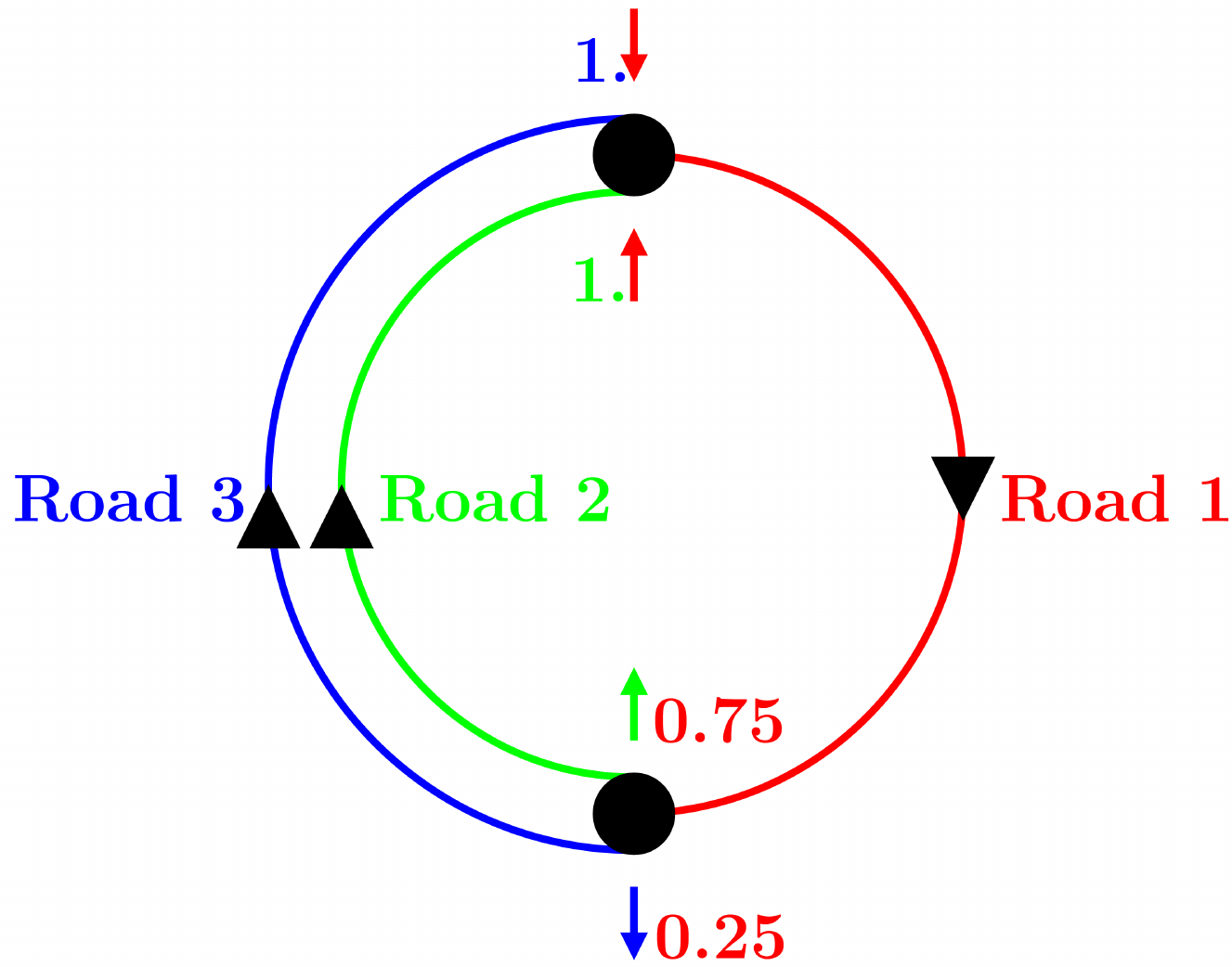}
\caption{Simple network.}
\label{obr_Network}
\end{figure}

\subsection{Simple network}\label{Subsection_simple_network}

Now we demonstrate how the method performs on networks. We define the simple network from Figure \ref{obr_Network}. This network is closed, so the total number of cars is conserved, by Theorem \ref{veta_properitoes_of_solution3}, since there is no inflow/outflow at artificial boundaries. We have three roads and two junctions. The length of all roads is $1$. At the first junction we have one incoming road and two outgoing roads. At the second junction we have the opposite situation. We consider a different distribution of cars to the two outgoing roads at the first junction: $\tfrac{3}{4}$ go from the first road to the second and $\tfrac{1}{4}$ from the first road to the third. This corresponds to the traffic distribution matrix $A_1=\T{[0.75, 0.25]}$. At the second junction, we simply take $A_2=[1,1]$. We note that $A_2$ does not satisfy the technical condition (C) from Section 5.1 in \cite{Networks}, cf. \cite[Remark 5.1.6]{Networks}, which ensures the existence of the fluxes from Definition \ref{def_network_solution}. Thus unlike the presented numerical scheme, the approach from \cite{Networks}, \cite{RKDG} cannot compute this example as is, but needs to introduce additional parameters into the problem. 

We define different initial conditions for each road. The initial condition for the first road is defined by
\begin{align*}
\rho_0(x)=
\begin{cases}
5x-1.5,\qquad &x\in[0.3,0.5],\\
-5x+3.5,\qquad &x\in[0.5,0.7],\\
0,\qquad &\text{otherwise},
\end{cases}
\end{align*}
which is a piecewise linear `bump'. The second and third road has a constant density of $0.4$, cf. Figure \ref{obr_Network_result_0}. The total number of cars in the whole network is $1$. We use the Greenshields model on all roads. The step size is $\tau=10^{-4}$ and the number of elements is $N=100$ on each road.

\begin{figure}[t!]\centering
\subfloat[$t=0$.]{\label{obr_Network_result_0}
\includegraphics[height=1.17in]{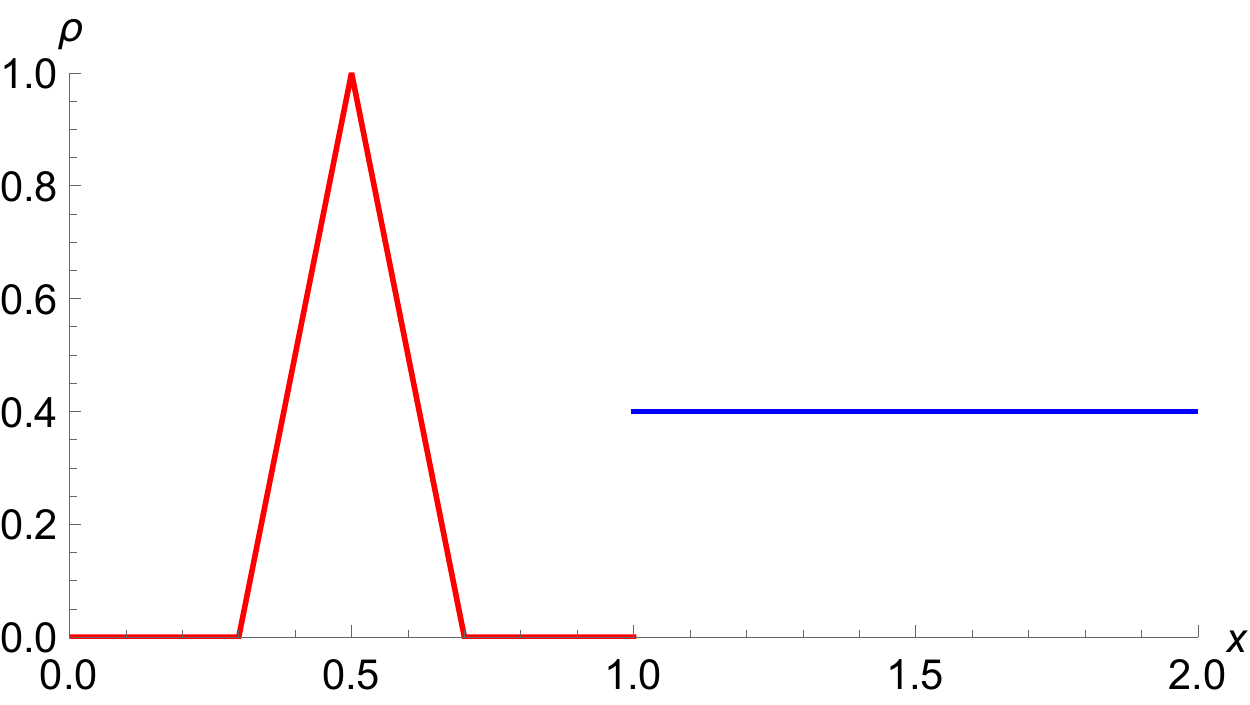}}
\hspace{5pt}
\subfloat[$t=0.2$.]{\includegraphics[height=1.17in]{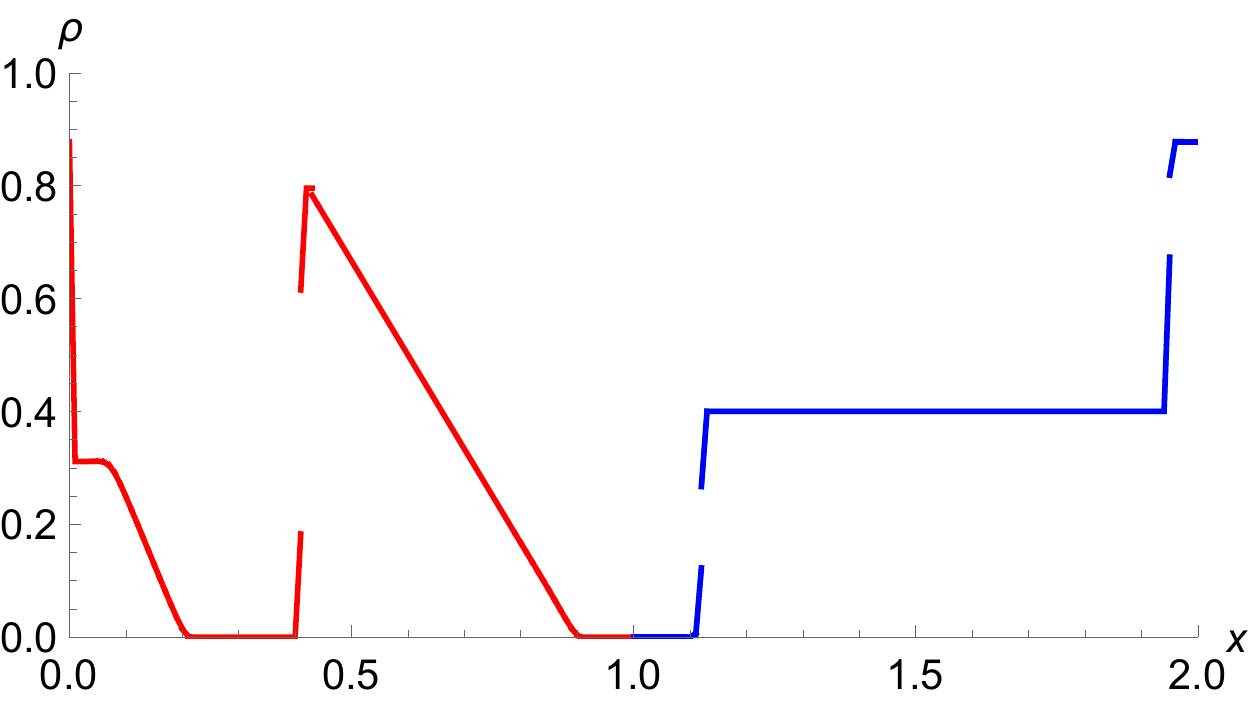}}
\hspace{5pt}
\subfloat[$t=0.4$.]{\includegraphics[height=1.17in]{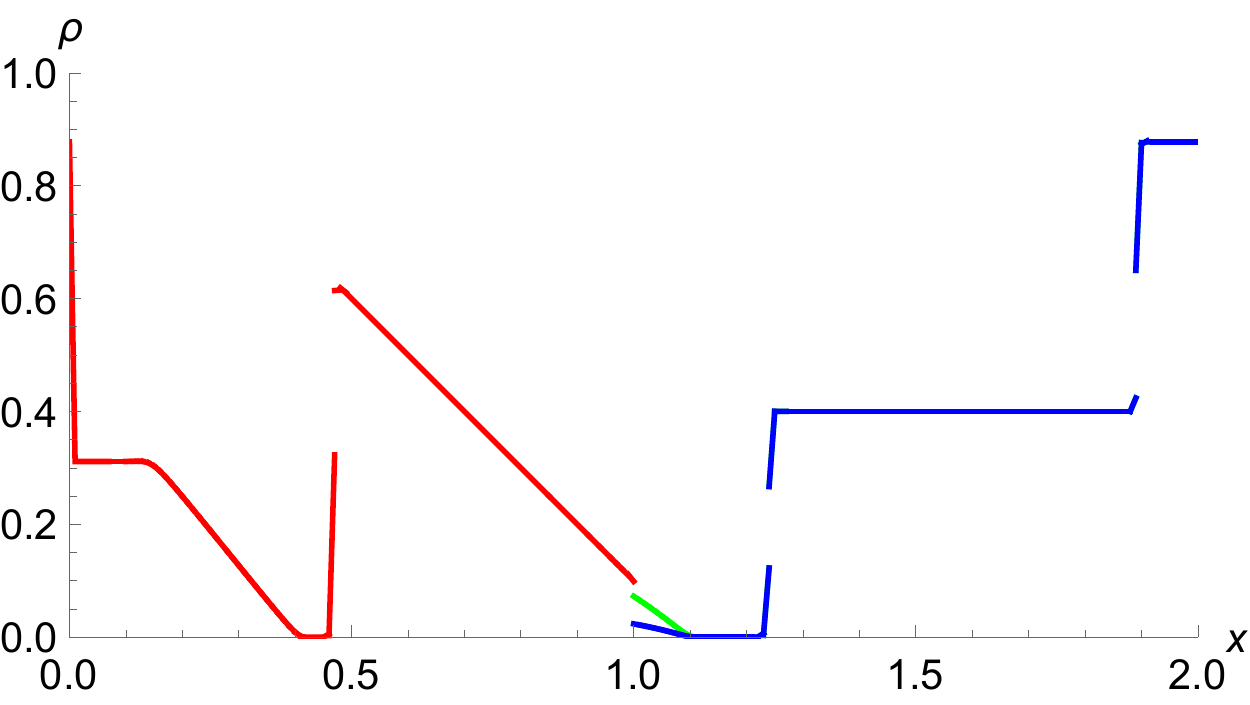}}
\hspace{5pt}
\subfloat[$t=0.6$.]{\includegraphics[height=1.17in]{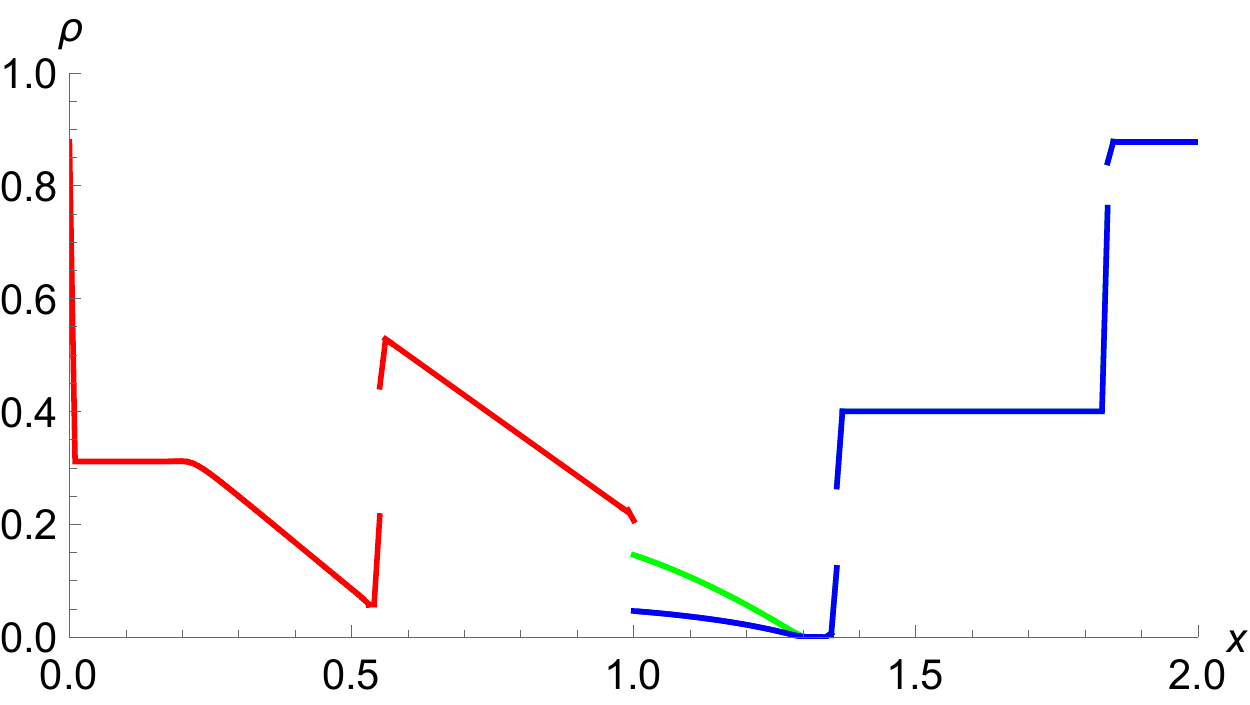}}
\hspace{5pt}
\subfloat[$t=0.8$.]{\includegraphics[height=1.17in]{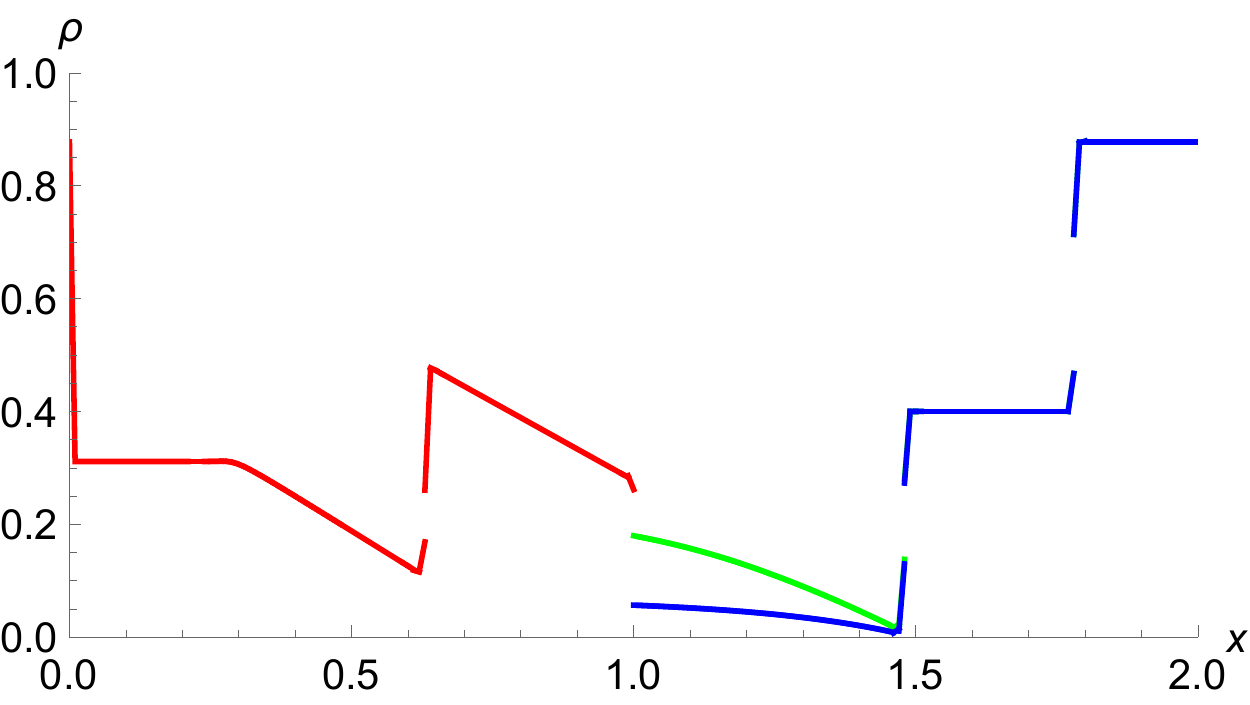}}
\hspace{5pt}
\subfloat[$t=1$.]{\includegraphics[height=1.17in]{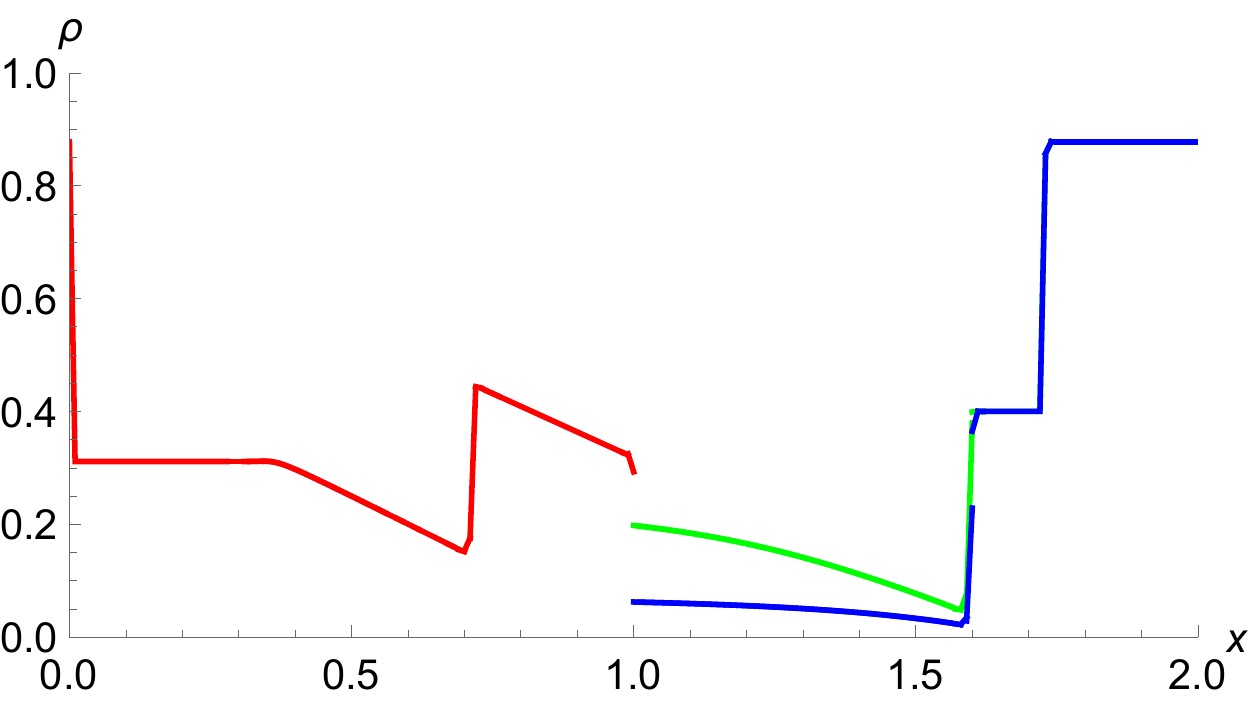}}
\hspace{5pt}
\subfloat[$t=1.5$.]{\label{obr_Network_result_6}\includegraphics[height=1.17in]{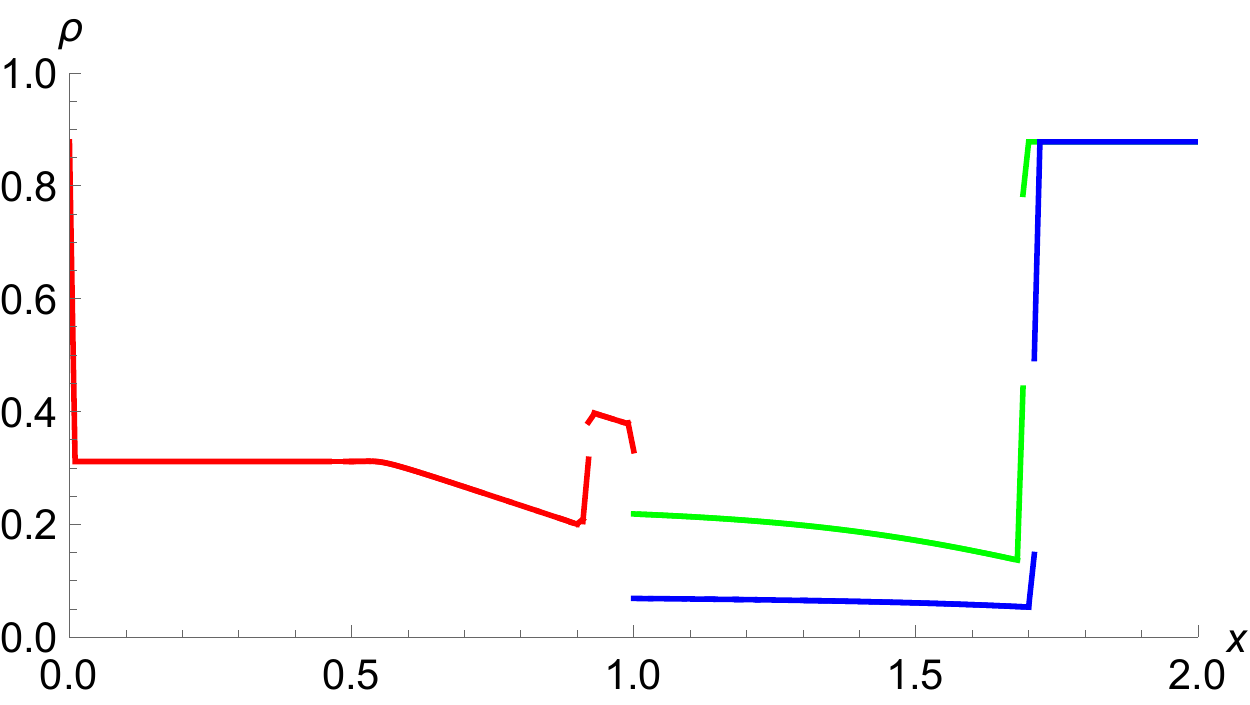}}
\hspace{5pt}
\subfloat[$t=2$.]{\label{obr_Network_result_7}\includegraphics[height=1.17in]{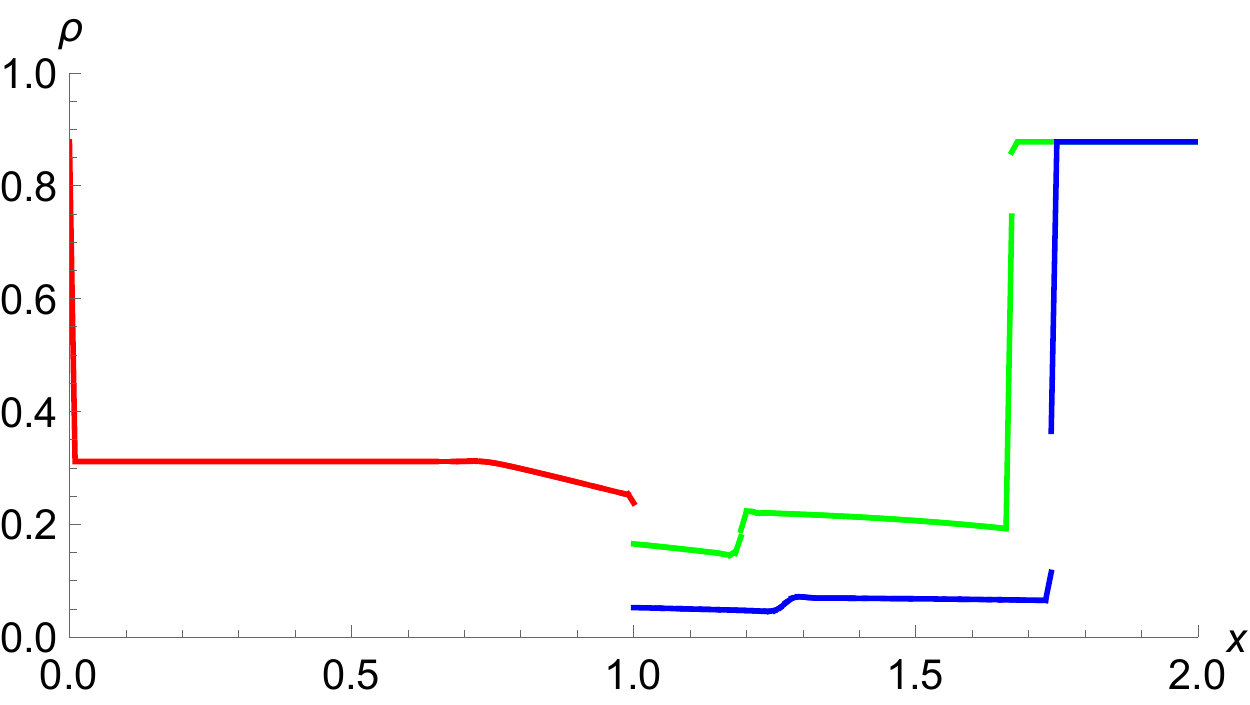}}
\hspace{5pt}
\subfloat[$t=3$.]{\label{obr_Network_result_8}\includegraphics[height=1.17in]{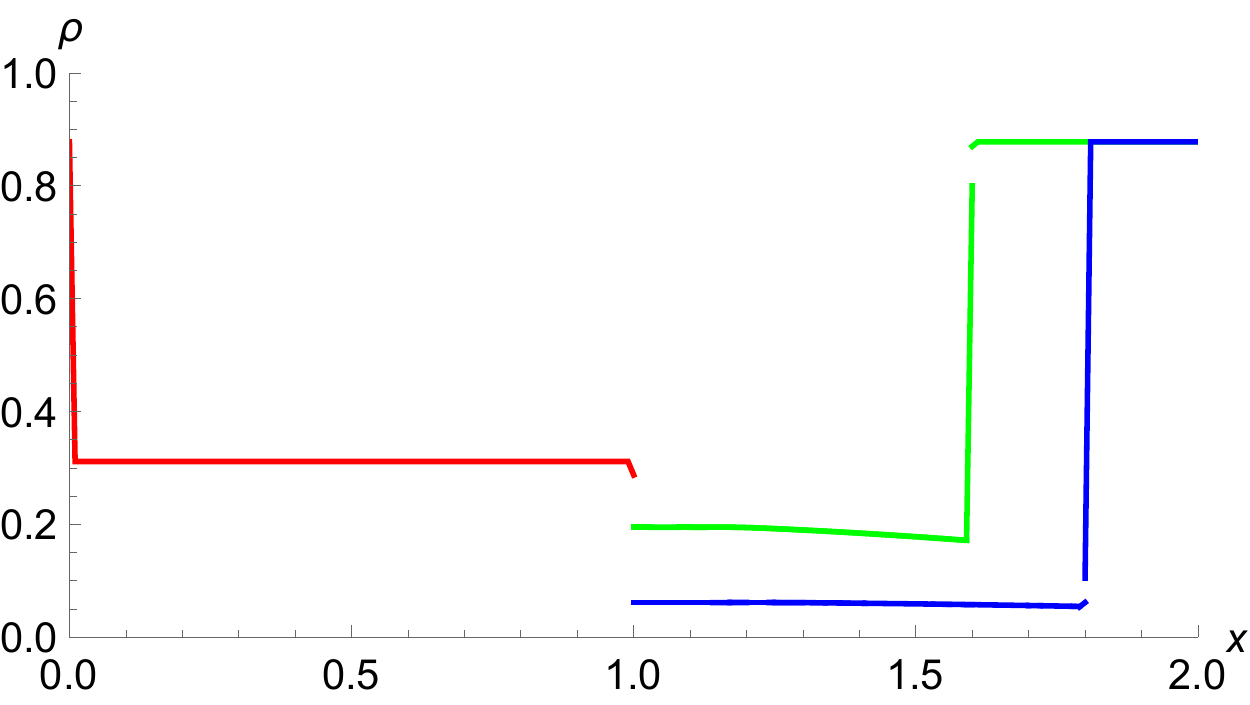}}
\caption{Network with \textcolor{red}{Road 1}, \textcolor{green}{Road 2} and \textcolor{blue}{Road 3}.}
\label{obr_Network_result}
\end{figure}

We can see the results in Figure \ref{obr_Network_result}. Road 1 distributes the traffic density between the other roads. We have too many cars at the second junction, where we have two incoming roads. Thus, we create a traffic congestion on Road 2 and Road 3. We can observe the transport and distribution of the bump from the first road through the first junction in Figures \ref{obr_Network_result_6} and \ref{obr_Network_result_7}. The result converges to a stationary solution. The traffic density in Figure \ref{obr_Network_result_8} is close to the stationary solution. The total amount of cars is conserved.

We have tested the method on much larger networks, where we are not limited by the number of incoming or outgoing roads at junctions. However in such cases the visualization of the results using density plots on individual roads, as in this paper, is impractical and confusing. For this purpose, other means of visualization must be implemented, such as maps of the network with individual roads colored by density magnitude using a suitable color palette. This remains for future work.

\subsection{Simple network -- comparison with the maximum possible fluxes}\label{sec_num_exp_comparison}
We consider the same network as in Section \ref{Subsection_simple_network} with different initial conditions:
\begin{align*}
\rho_{0,1}(x)&=
\begin{cases}
0,&x\in[0,0.5],\\
1,&x\in[0.5,1],
\end{cases}
\qquad
\rho_{0,2}(x)=
\begin{cases}
5x-1.5,&x\in[0.3,0.5],\\
-5x+3.5,&x\in[0.5,0.7],\\
0,&\text{otherwise},
\end{cases}
\\
\rho_{0,3}(x)&=
\begin{cases}
1,&x\in[0,0.5],\\
0,&x\in[0.5,1],
\end{cases}
\end{align*}
where $\rho_{0,i}$ is the initial condition on road number $i$.

\begin{figure}\centering
\subfloat[$t=0$.]{\includegraphics[height=1.1in]{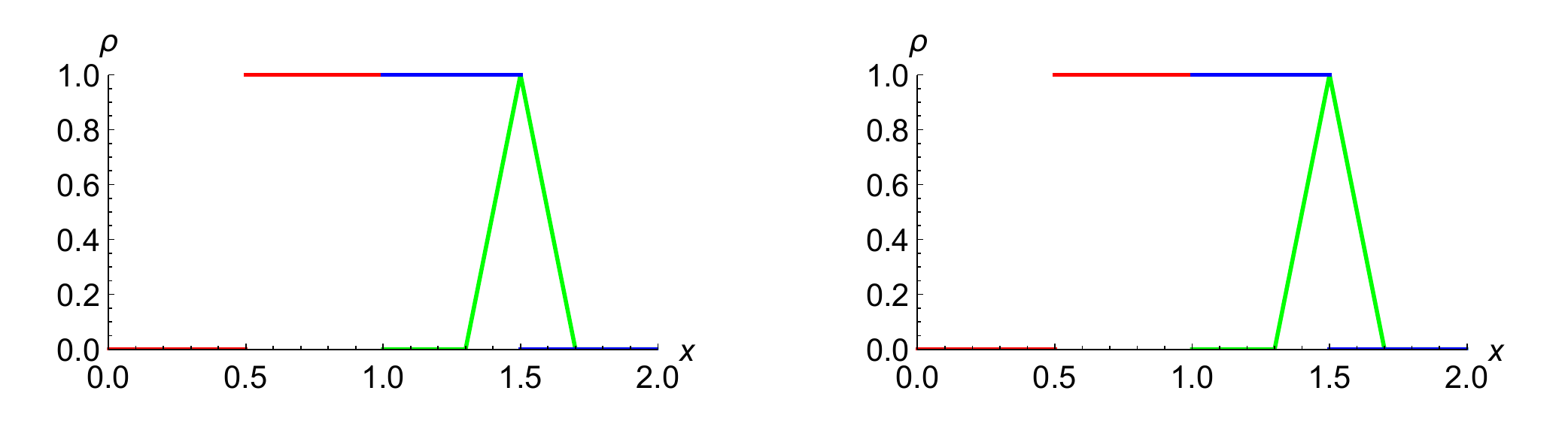}}
\hspace{0pt}
\subfloat[$t=0.25$.]{\label{obr_Network_Comparison_result_1}\includegraphics[height=1.1in]{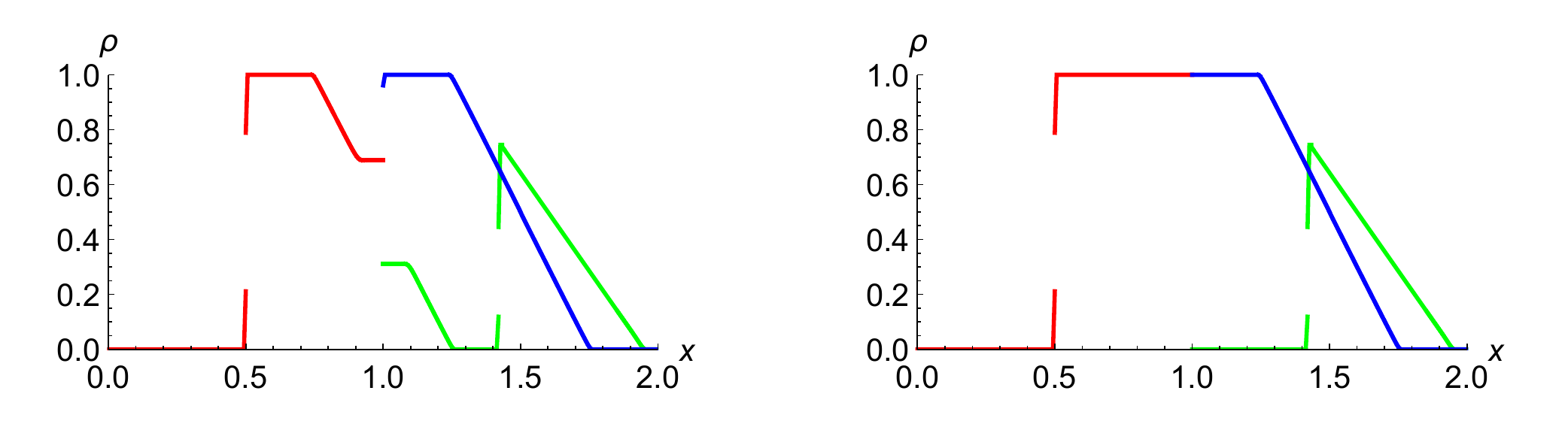}}
\hspace{0pt}
\subfloat[$t=0.5$.]{\label{obr_Network_Comparison_result_2}\includegraphics[height=1.1in]{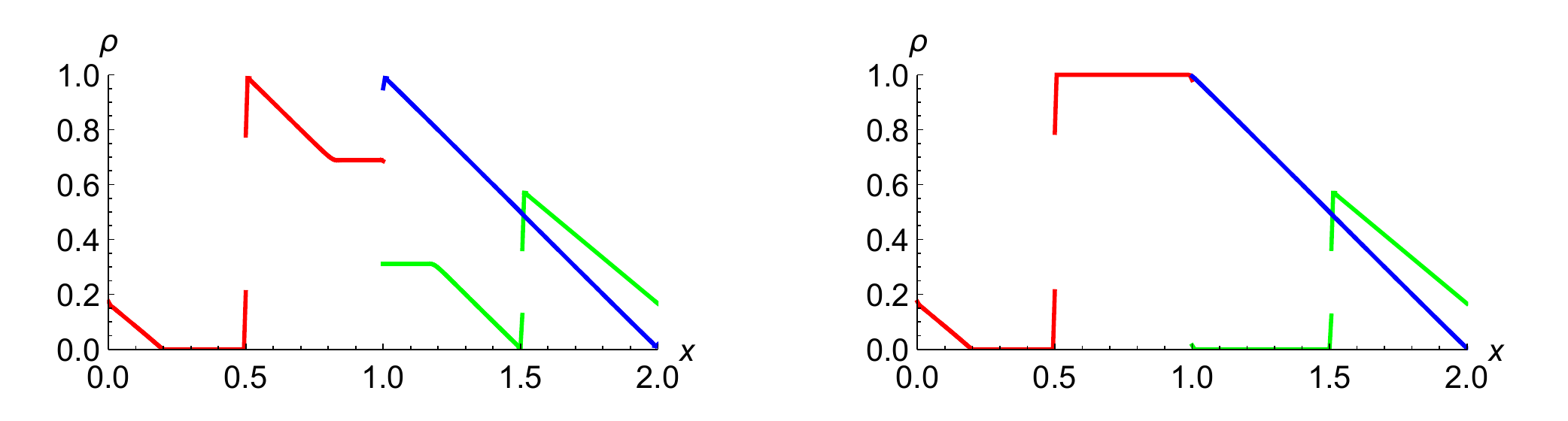}}
\hspace{0pt}
\subfloat[$t=0.75$.]{\label{obr_Network_Comparison_result_3}\includegraphics[height=1.1in]{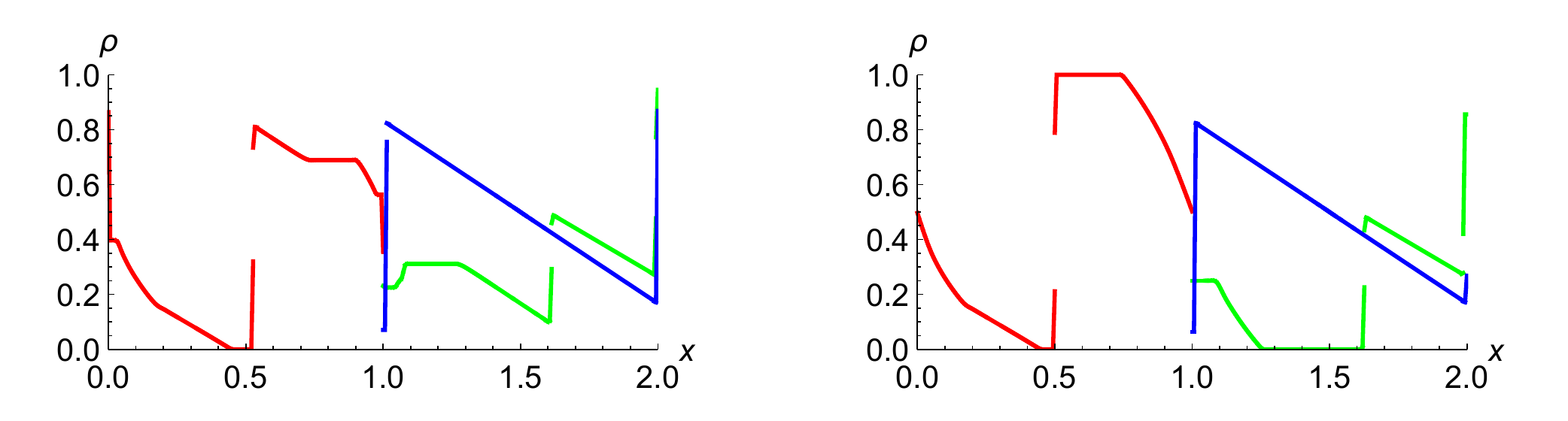}}
\hspace{0pt}
\subfloat[$t=1$.]{\label{obr_Network_Comparison_result_4}\includegraphics[height=1.1in]{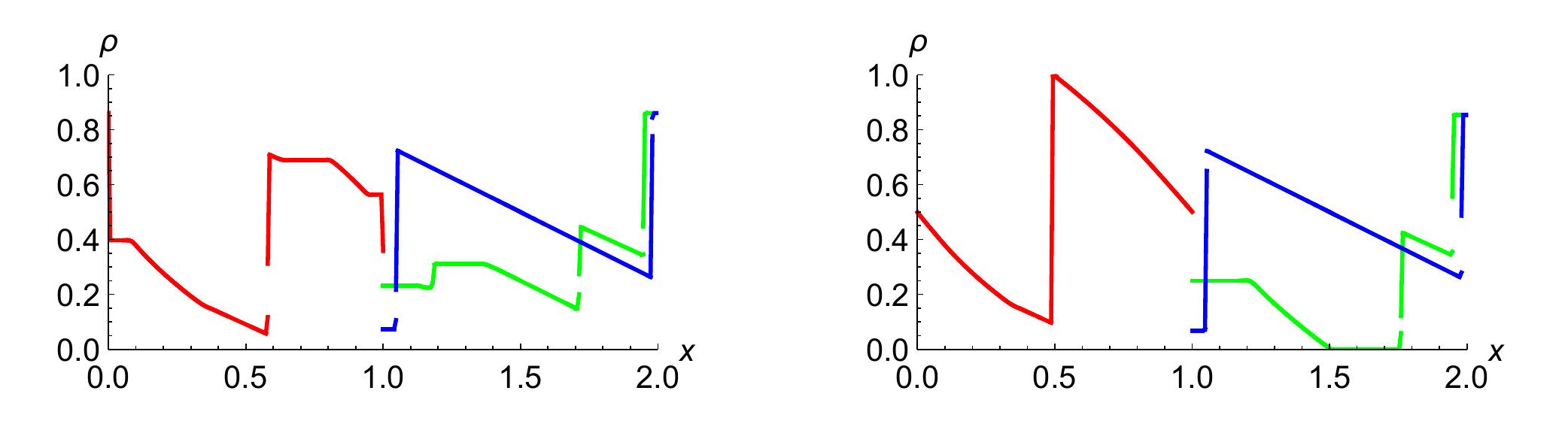}}
\hspace{0pt}
\subfloat[$t=20$.]{\label{obr_Network_Comparison_result_5}\includegraphics[height=1.1in]{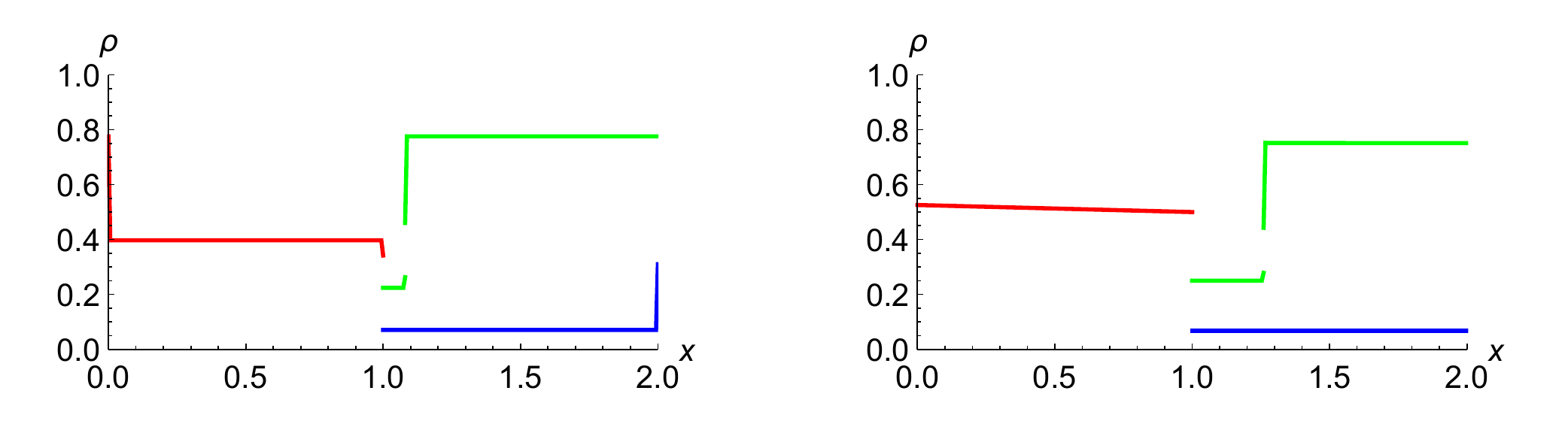}}
\caption{Comparison of network with \textcolor{red}{Road 1}, \textcolor{green}{Road 2} and \textcolor{blue}{Road 3}. Left column -- numerical flux from Section \ref{Subsection_Fluxes_junctions}. Right column -- maximal flux from Definition \ref{def_network_solution}.}
\label{obr_Network_Comparison_result}
\end{figure}

We compare our approach with that of \cite{RKDG} which uses the maximum possible flux from Definition \ref{def_network_solution}. In both approaches we use the Lax--Friedrichs flux and the explicit Euler method. A right of way parameter $q$ must be prescribed for the junction with two incoming roads in the case of the maximum possible flux. We use $q=0.5$, so the roads are equal. In our approach, we do not have a defined right of way, so the roads are equal as well.

We can see the comparison in Figure \ref{obr_Network_Comparison_result}. Our approach is in the left column while the approach using the maximum possible flux is in the right column. We point out the different behavior in both junctions.

First, we notice the first junction with one incoming and two outgoing roads, i.e. $x=1$ in the figures. As we mention in Section \ref{Subsection_Fluxes_junctions}, the maximum possible flux through the junction at the time $t\in[0,0.5]$ is zero because one of the outgoing roads (Road 3) reaches the maximal traffic density, cf. Figure \ref{obr_Network_Comparison_result_1} and \ref{obr_Network_Comparison_result_2}. Our approach has nonzero traffic flow through this junction at the time $t\in[0,0.5]$ because the numerical flux is nonzero between Road 1 and Road 2 allowing the cars to go from Road 1 to Road 2. For times $t>0.5$, the maximal traffic density is not attained on Road 3 and the  traffic flow is nonzero through the junction in both cases, cf. Figure \ref{obr_Network_Comparison_result_3}, \ref{obr_Network_Comparison_result_4} and \ref{obr_Network_Comparison_result_5}. If we compare both approaches, we see completely different results on Roads 1 and 2 while the results on Road 3 are almost identical.

Now we focus on the second junction with two incoming and one outgoing road, i.e. $x=0$ and $x=2$ in the figures. At first glance, there is no difference between the two approaches. Let's compare $\rho_1^{(R)}(0,1)$, i.e. the limit from the right of traffic density on the outgoing Road 1 at $x=0$ and $t=1$. Our approach gives us $\rho_1^{(R)}(0,1)\approx 0.4$ while the approach using the maximum possible flux gives us $\rho_1^{(R)}(0,1)\approx 0.5$, which is the maximal traffic flow. The reason for this difference is that we do not have a defined right of way in our approach. Road 2 and Road 3 push too many cars into the junction congesting it slightly. The approach using the maximum possible flux takes into account the whole situation and selects the best solution for both roads. From a real point of view, this approach could be viewed as simulating the behavior of communicating autonomous vehicles which optimize the traffic situation globally, while our approach could be interpreted as simulating the behavior of human drivers without the right of way.

Both approaches converge to stationary solutions which are not identical, see Figure \ref{obr_Network_Comparison_result_5}.

We would like to implement right of way into our approach and introduce it in future work.

\subsection{Traffic lights}
Finally, we apply the presented method to traffic on a junction with traffic lights. The advantage of our approach is that we are not strictly forced to use only full green or red for all outgoing roads, as discussed in Section \ref{sec_DG_Networks}. Our traffic flow at the junction allows us to choose from a large variety of traffic light combinations.

We define a junction with 4 incoming and 4 outgoing roads, see Figure \ref{obr_Lights}. The outgoing roads turn back and return to the junction. Roads 1 and 2 are the main roads. The maximal density on the main roads is $\rho_{\max,m}=2$ and the length is $L_m=0.5$. The initial condition for the main roads is defined as $\rho_{0,m}(x)=1.3$. Roads 3 and 4 are the side roads. The maximal density is $\rho_{\max,s}=1$ and the length is $L_s=0.4$. The initial condition for the side roads is defined as $\rho_{0,s}(x)=0.2$.

At the junction we use the traffic distribution matrix
\[
A=\begin{bmatrix}
0 &\ 0.75 &\ 0.4 &\ 0.45\\
0.8 &\ 0 &\ 0.5 &\ 0.4 \\
0.1 &\ 0.15 &\ 0 &\ 0.15 \\
0.1 &\ 0.1 &\ 0.1 &\ 0
\end{bmatrix}.
\]
We define three phases of traffic lights. In the first phase, traffic lights allow vehicles from Road 1 to drive to Road 2 or Road 3 and vehicles from Road 2 to drive to Road 1 or Road 4. The first phase lasts for $t_1=1$. In the second phase, traffic lights allow vehicles from Road 1 to drive to Road 4, vehicles from Road 2 to drive to Road 3, vehicles from Road 3 to drive to Road 2 and vehicles from Road 4 to drive to Road 1.  In the third phase, the traffic lights on Road 3 and Road 4 have full green signal. The second and third phase lasts for $t_2=0.5$. After each phase there are all red lights and this situation lasts for $t_r=0.05$. All three phases are periodically alternating.

The maximal velocity on each roads is $v_{\max}=0.5$. The maximal density at the junction is $\rho_{\max,j}=2$. For side roads, in order to accommodate the different maximal densities at the junction and side roads, we linearly interpolate the maximal density on the first and last elements of both side roads. We use Greenshields model. The time--step size is $\tau=10^{-4}$ and the length of each element is $h=\tfrac{1}{150}$.

We can see the results in Figure \ref{obr_Krizovatka_result}. The first phase is in Figures \ref{obr_Krizovatka_result_1} and \ref{obr_Krizovatka_result_2}. The second phase is in Figure \ref{obr_Krizovatka_result_3}. The third phase is in Figure \ref{obr_Krizovatka_result_5}. There are red lights on each road in Figure \ref{obr_Krizovatka_result_4}.

\label{num_exp_traffic_lights}
\begin{figure}[t!]\centering
\includegraphics[height=2.3in]{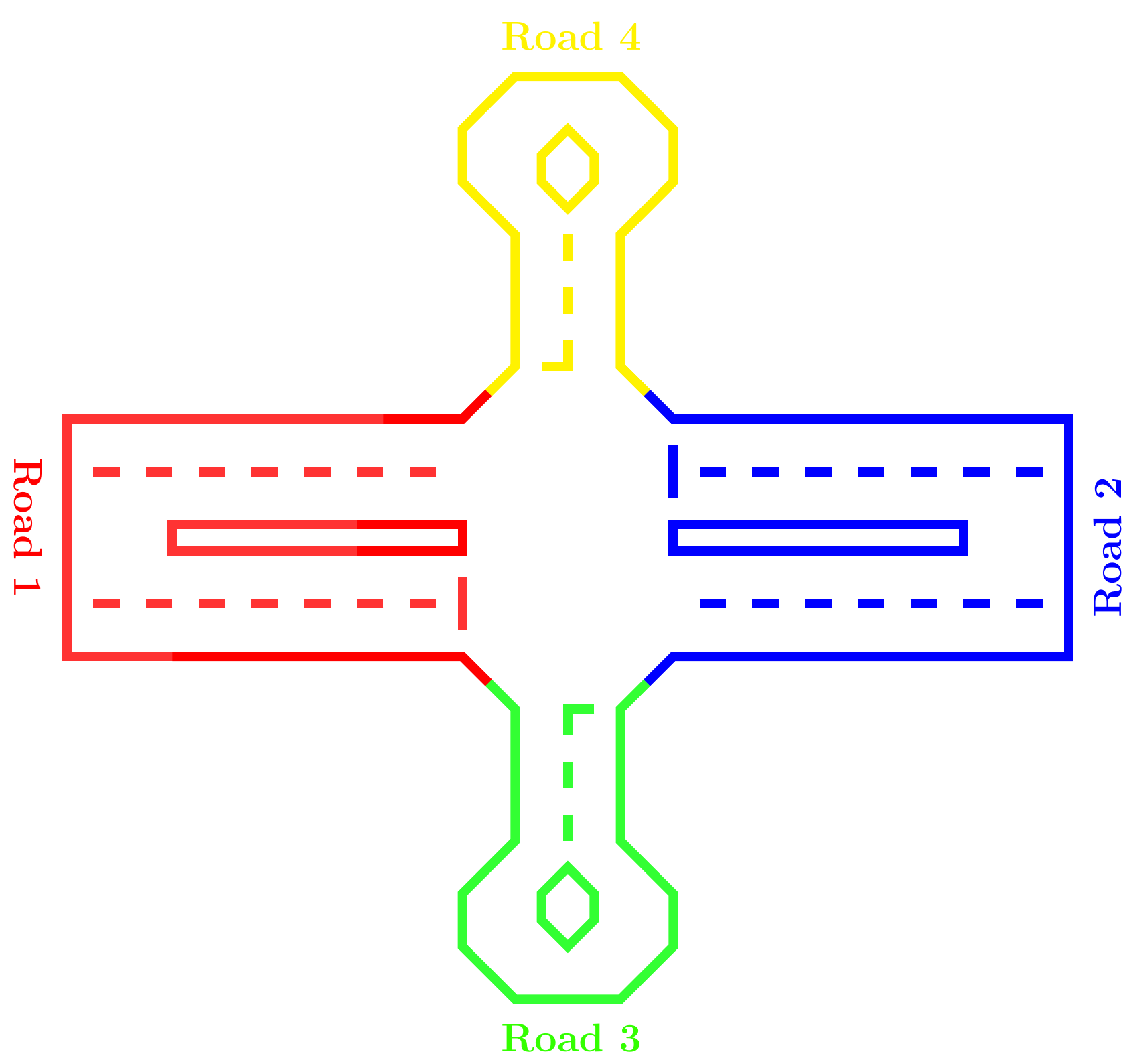}
\caption{Junction with traffic lights.}
\label{obr_Lights}
\end{figure}

\begin{figure}[t!]\centering
\subfloat[$t=0$.]{\includegraphics[height=1.99in]{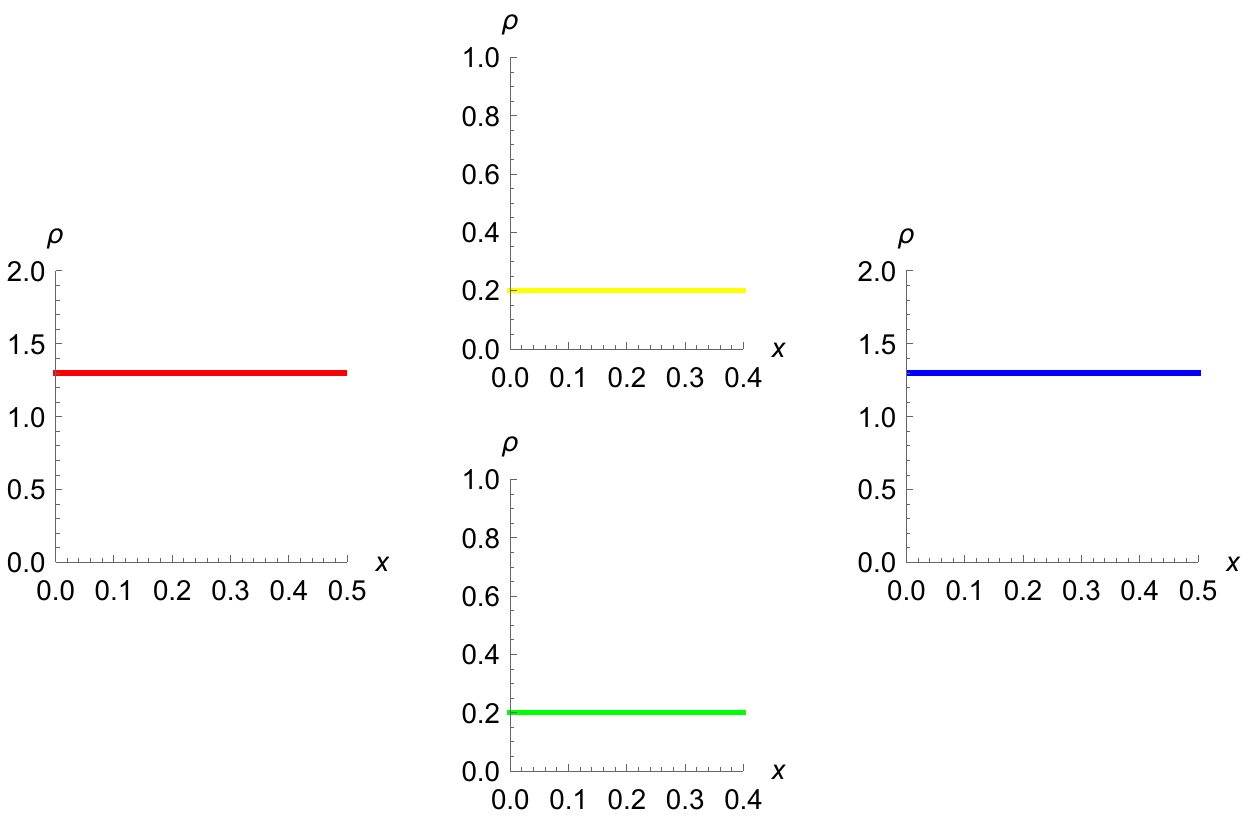}}
\hspace{5pt}
\subfloat[$t=0.4$.]{\label{obr_Krizovatka_result_1}\includegraphics[height=1.99in]{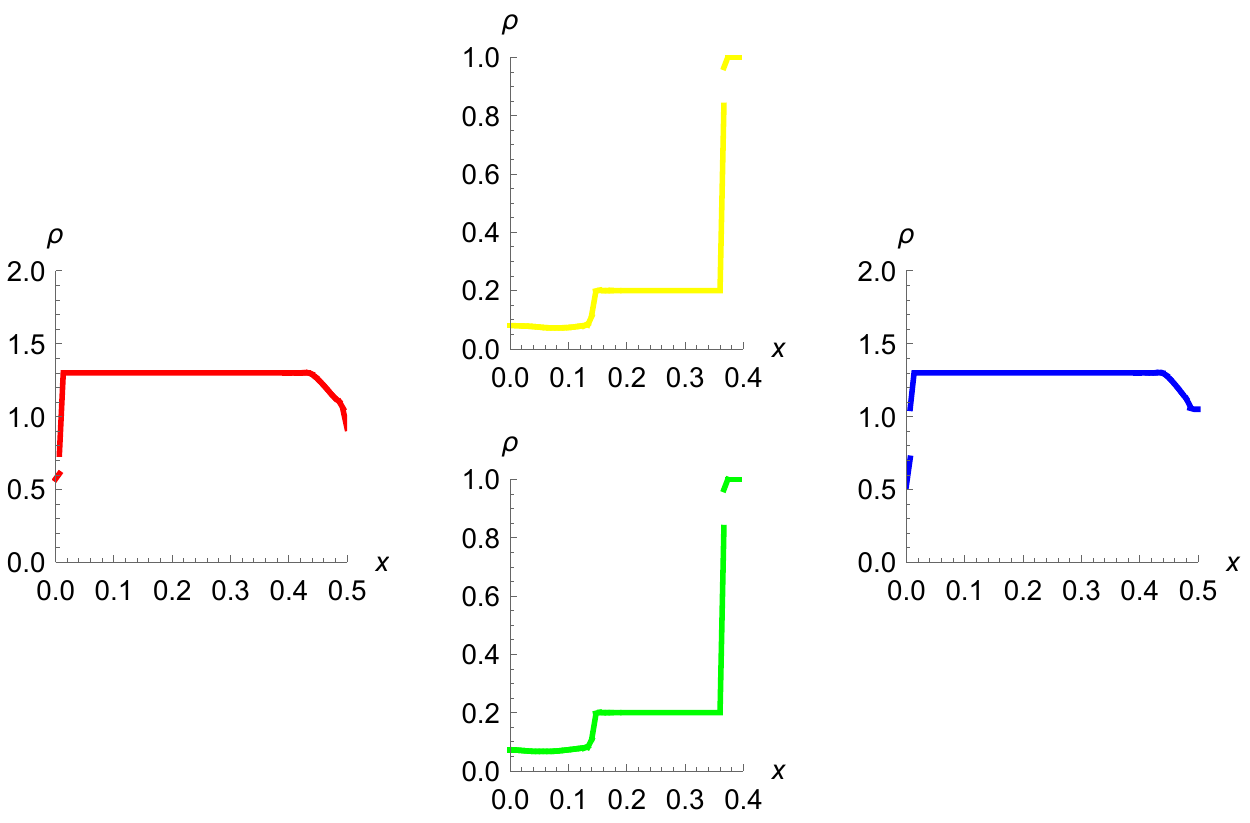}}
\hspace{5pt}
\subfloat[$t=0.8$.]{\label{obr_Krizovatka_result_2}\includegraphics[height=2.1in]{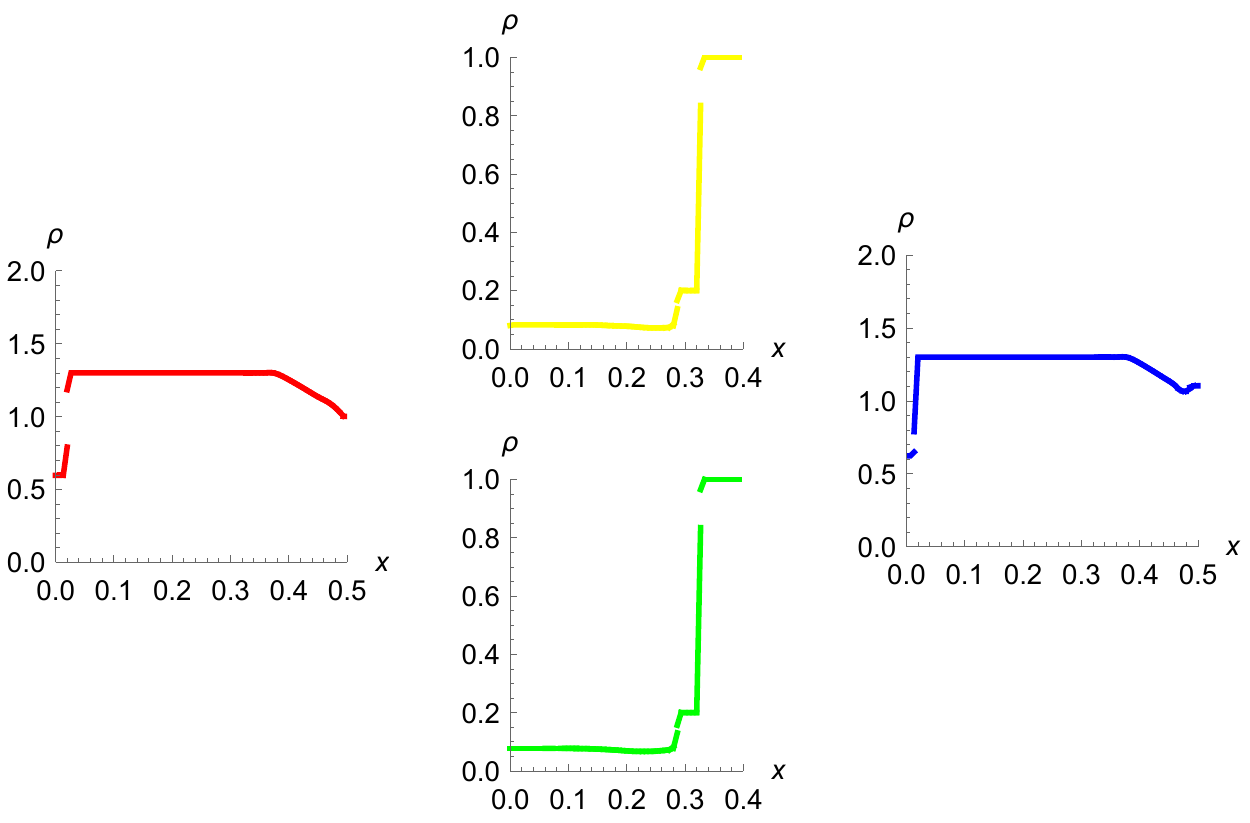}}
\hspace{5pt}
\subfloat[$t=1.2$.]{\label{obr_Krizovatka_result_3}\includegraphics[height=2.1in]{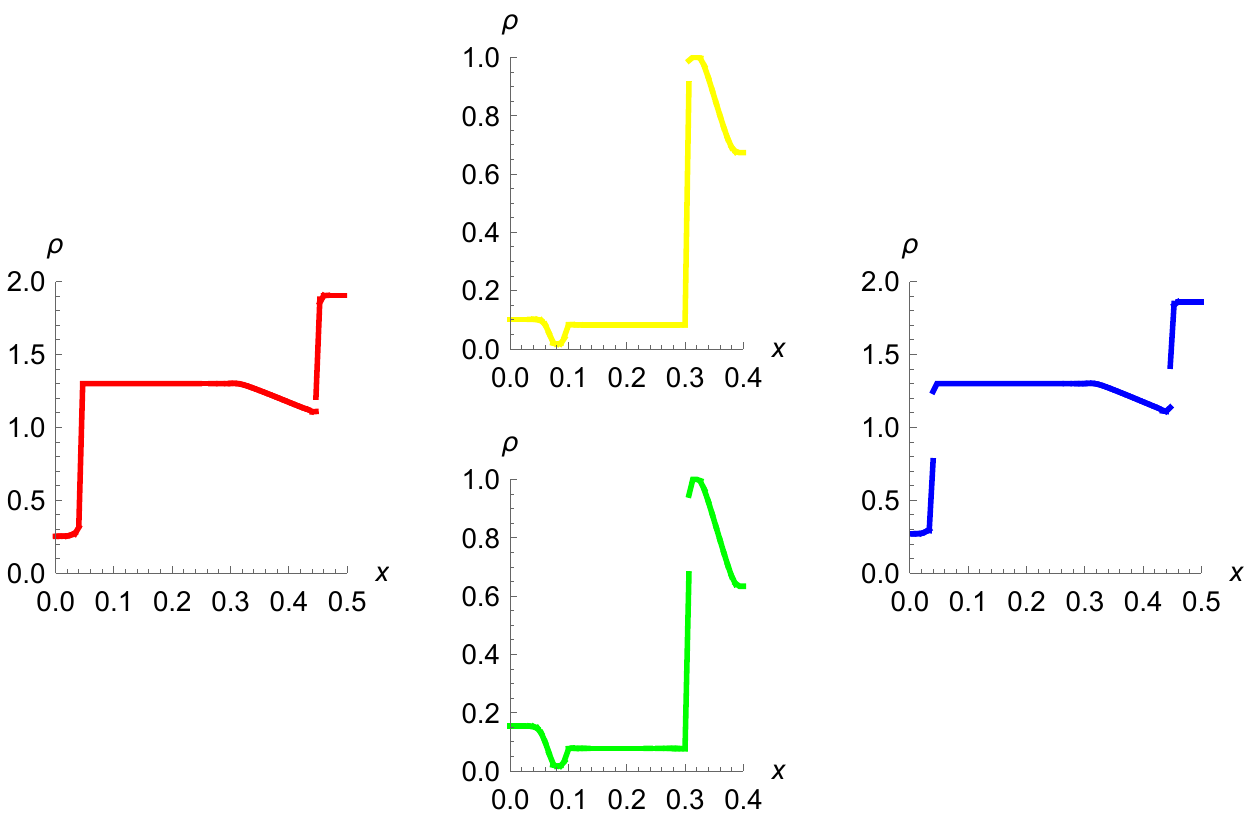}}
\hspace{5pt}
\subfloat[$t=1.6$.]{\label{obr_Krizovatka_result_4}\includegraphics[height=2.1in]{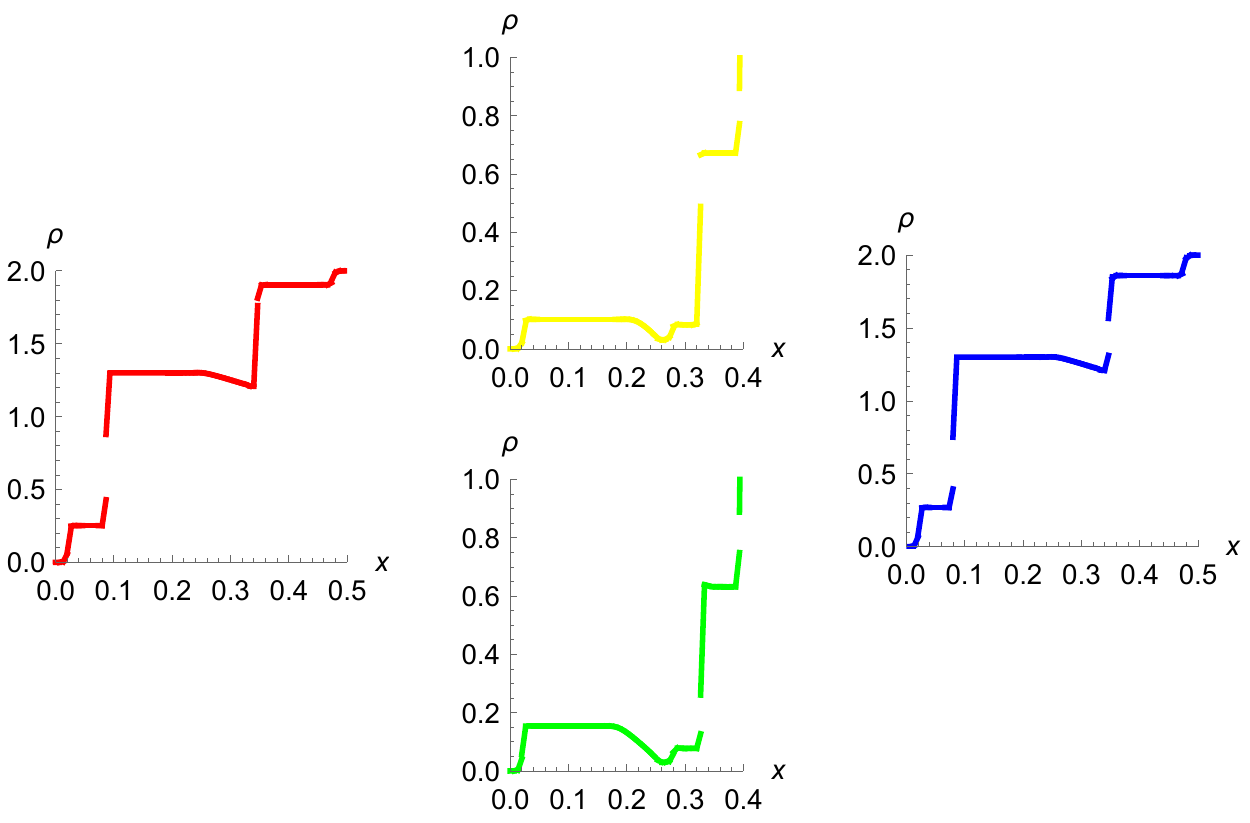}}
\hspace{5pt}
\subfloat[$t=2$.]{\label{obr_Krizovatka_result_5}\includegraphics[height=2.1in]{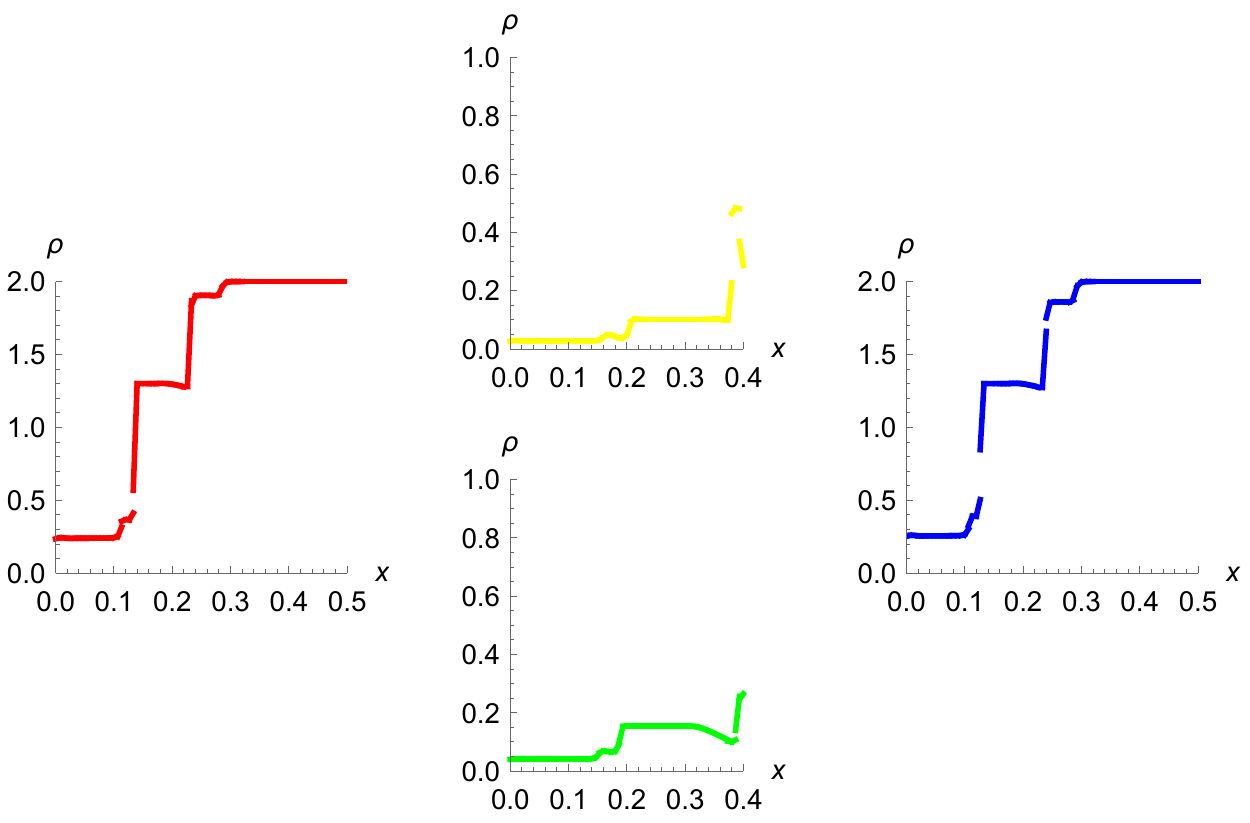}}
\caption{Junction with \textcolor{red}{Road 1}, \textcolor{blue}{Road 2}, \textcolor{green}{Road 3} and \textcolor{yellow}{Road 4}.}
\label{obr_Krizovatka_result}
\end{figure}

\section*{Conclusion}
We have demonstrated the numerical solution of macroscopic traffic flow models on networks using the discontinuous Galerkin method. On individual roads, we use the Lax-Friedrichs numerical flux, while on junctions, we construct a new numerical flux based on the preferences of the drivers. We compare our approach with the paper \cite{RKDG} by \v Cani\'c, Piccoli, Qiu and Ren, where Runge-Kutta methods are used along with a different choice of numerical fluxes at junctions. We discuss the differences between the two approaches, where that of \cite{RKDG} corresponds to single-lane roads with a strict enforcement of a priori traffic distribution, while the presented approach corresponds to having dedicated turning-lanes and/or flexibility of the drivers' preferences in extreme situation such as congestions. Moreover, the presented construction of the traffic flux at junctions allows the simulation of arbitrary traffic light combinations. In the future works, we would like to implement right of way rules (with regard to main and side roads) into the numerical flux and introduce true multi-lane roads with overtaking into the model.

\bibliographystyle{spmpsci}      
\bibliography{Vacek_bibliography}   

\end{document}